\documentclass[10pt,reqno]{article}
\usepackage{amssymb}
\usepackage{amsfonts}
\usepackage{amsmath}
\usepackage{amsthm}
\usepackage{remark}
\usepackage{enumerate}
\usepackage{xypic}
\usepackage{graphicx}
\usepackage{hyperref}

\newtheorem{thm}{Theorem}[section]
\newtheorem{theorem}[thm]{Theorem}
\newtheorem{lemma}[thm]{Lemma}
\newtheorem{corollary}[thm]{Corollary}
\newtheorem{proposition}[thm]{Proposition}
\newremark{definition}[thm]{Definition}
\newremark{remark}[thm]{Remark}
\newremark{example}[thm]{Example}
\newremark{notation}[thm]{Notation}
\newremark{algo}[thm]{Algorithm}
\newremark{Question}[thm]{Question}
\newremark{emp}[thm]{}
\numberwithin{equation}{section}

\newcommand{\Ast}{\widetilde{A}_{(\tilde{s})}}

\newcommand{\bl}{\mathrm{bl}}
\newcommand{\ord}{\operatorname{ord}}

\newcommand{\Spec}{\operatorname{Spec}}
\newcommand{\Proj}{\operatorname{Proj}}

\newcommand{\Z}{\mathbb Z}
\newcommand{\F}{\mathbb F}

\newcommand{\N}{\mathbb N}
\newcommand{\Q}{\mathbb Q}

\newcommand{\cO}{\mathcal O}
\newcommand{\cE}{\mathcal E}
\newcommand{\PP}{\mathbb P}
\newcommand{\m}{\mathfrak m}
\newcommand{\n}{\mathfrak n} 
\newcommand{\p}{\mathfrak p}
\newcommand{\q}{\mathfrak q}
\newcommand{\chara}{\mathrm{char}}
\newcommand{\Gr}{\mathrm{Gr}}

\newcommand{\Frac}{\mathrm{Frac}}

\title{\bf Desingularization of double covers of  regular surfaces}
\author{Qing Liu\thanks{I would like to thank INRIA Sud-Ouest and the team Canari for
  their hospitality during the preparation of a part of this work.}
}
\date{}

\begin{document}

\maketitle


\begin{abstract} Let $Z$ be a noetherian integral excellent regular scheme of
  dimension $2$. Let $Y$ be an integral normal scheme endowed with a
  finite flat morphism $Y\to Z$ of degree $2$. We give a
  description of Lipman's desingularization of $Y$ by explicit equations, 
leading to a desingularization algorithm for $Y$. 
\end{abstract}

\tableofcontents

\begin{section}{Introduction} 

The existence of the desingularization of excellent \emph{surfaces}
(reduced noetherian excellent schemes of dimension 2) is well known from the
theoretical point of view. However, in practice, the 
implementation of an algorithm of desingularization in full
generalities is rather complicated, see \cite{CJS} and \cite{FRS} for an
overview of various approaches. In view of applications in arithmetic
geometry, a lot of work also have been 
done for smooth projective curves $C$ over a discrete
valuation field $K$. The aim is then to find a projective regular
scheme over the valuation ring $\cO_K$ of $K$, with generic fiber isomorphic to
$C$, see \cite{KW}, \cite{Mu} and the references therein, when $C$
is a cyclic cover of $\PP^1_K$, mostly assuming the residue
characteristic of $K$ to be prime to the order of the cover. 
For hypersurfaces in the affine plane $\mathbb A^2$ over $\cO_K$,
see also  \cite{TDo} with some powerful applications in number
theory. 

In this work we deal with a special type of surfaces for which
Lipman's method works very well and is completely explicit. 

\begin{theorem}[Lipman, \cite{Lip}] Let $Y$ be a noetherian integral
  normal excellent surface. Consider a sequence of birational morphisms 
  \begin{equation}
    \label{eq:lipman}
... \to Y_n \to \dots \to Y_1 \to Y_0:=Y     
  \end{equation}
where each $Y_{i+1} \to Y_i$ is the normalization of the blowing-up of
$Y_i$ along  some singular closed (reduced) point. Then the sequence is
necessarily finite.  
\end{theorem}

The surfaces we are interested in are double covers of regular
surfaces. For simplicity we only work with noetherian excellent
schemes ({\it e.g.}, algebraic varieties over a field, schemes of
finite type over Dedekind domains of characteristic $0$ or complete
discrete valuation rings, see \cite{LB}, \S 8.2.3). 

\begin{definition} \label{def-dc} 
Let $Z$ be an integral noetherian regular scheme.
We define  a \emph{double cover of $Z$} as a reduced  
scheme $Y$ endowed with a finite flat morphism 
of degree $2$
$$\psi : Y\to Z.$$
We call $Y$ a \emph{normal double cover of $Z$} if moreover $Y$
is integral and normal. 
\end{definition}

\begin{example} \label{exd} Let $C$ be a hyperelliptic curve over $\Q$, with
  the canonical degree $2$ morphism $C\to \PP^1_\Q$. Let
  $Z=\PP^1_{\Z}$ and let $W$ be the normalization of $Z$ in the
  function field $K(C)$ of $C$. Then $Z$ is a regular excellent surface,
  and $W\to Z$ is a normal double cover of $Z$. The surface $W$ is a
  so called Weierstrass model of $C$. It is singular in general. A
  resolution
  of singularities on $W$ gives a projective regular model of $C$ over
  $\Z$. 
\end{example}

\begin{example}  \label{exd_2} More generally, let $Z$ be an
  integral noetherian excellent surface.
  Let $L$ be a (possibly inseparable) 
  quadratic extension of the function field $K(Z)$ of $Z$ and
  let $Y$ be the normalization of $Z$ in $L$.  Then $Y\to Z$ is a
  normal double cover. Indeed, $Y\to Z$ is finite because $Z$ is
  excellent. The scheme $Y$ a noetherian integral normal surface, 
  hence Cohen-MaCaulay, hence flat over $Z$ because the latter is regular.  
\end{example}

Now come back to the general situation. The key point in our situation
is that if we start with a normal double cover $Y\to Z$ of a regular
surface, then in any sequence like \eqref{eq:lipman}, all the $Y_i$'s
are again normal double covers of regular surfaces. 
More precisely (Theorem~\ref{desing-main}),
we have a canonical  commutative diagram
\[
\xymatrix{
  Y_n \ar[r] \ar[d] &  Y_{n-1} \ar[r] \ar[d]& \cdots  \ar[r] &  Y_1
  \ar[d] \ar[r] &  Y_0=Y \ar[d] \\ 
   Z_n \ar^{\bl_{q_{n-1}}}[r]         &  Z_{n-1}
   \ar^{\bl_{q_{n-2}}}[r] 
   & \cdots  \ar^{\bl_{q_1}}[r] & Z_1 
   \ar^{\bl_{q_0}\hskip 4pt }[r] &   \  Z_0:=Z 
}
\]
where the vertical arrows are normal double covers of regular
surfaces, and the lower horizontal arrows are blowing-ups of  
(reduced) closed points $q_i\in Z_i$. Each morphism $Y_{i+1}\to Y_i$ is a
\emph{normalized blowing-up}, meaning that we blow-up the 
singular point (with the reduced structure) of $Y_i$ lying over $q_i$,
then we take the normalization. 

Moreover, we can give explicit equations for the $Y_i$'s as follows. 
First one can cover $Z_0$ by small affine open subsets $V$'s over which
$Y_0$ is defined by  a monic polynomial of degree $2$
(Lemma~\ref{CM-dc}):  
\begin{equation}
  \label{eq:first_eq}
y^2+ay+b=0, \quad a, b\in \cO_{Z}(V).     
\end{equation}
If $Z_1\to Z_0$ is the blowing-up of a point $q_0\in Z_0$, and if 
$\lambda_{q_{0}}(Y_{0})$ is the multiplicity of $Y_{0}$ at the point
lying over $q_{0}$ (Definitions~\ref{def-lambda} and
\ref{double-cov}), we set $r=[\lambda_{q_0}(Y_0)/2]$.  
Let $s,t$ be a system of coordinates of
$Z_0$ at $q_0$. Then $Y_{1}$ is covered by the affine schemes  
defined respectively by
$$
y_1^2+(a/s^r)y_1+(b/s^{2r})=0, \quad y_1=y/s^{r}
$$
and 
$$
z_1^2+(a/t^r)z_1+(b/t^{2r})=0, \quad z_1=y/t^r
$$ 
(Proposition~\ref{lemma2.12}). And we continue in the same way for the
$Y_i$'s. 
\medskip

The computation of the multiplicity $\lambda_q(Y)\in \N$ is
straightforward when $2\in \cO_{Z,q}^*$. Indeed we can reduce to $a=0$
and $b(q)=0$. Then $\lambda_q(Y)$ is just the order of $b$ in the regular
local ring $\cO_{Z,q}$ (Proposition~\ref{lambda-f}(2)). If $2$ is not
invertible, there is no 
canonical choice of the equation~\eqref{eq:first_eq}. 
We give a simple algorithm (Lemma~\ref{lambdaz}) to determine
$\lambda_q(Y)$ in the case. 
\medskip

As a byproduct of Theorem~\ref{desing-main} 
we prove the existence of a simultaneous resolution of singularities
for covers of degree $2$ of excellent normal surfaces (Corollary~\ref{sim}). 
\medskip

When $2$ is invertible in $\cO_Z$, the sequence $(Z_i)_{i\geq 0}$
is classical and  
consists in making the branch locus of $Y\to Z$ regular. 
In the case $Z=\PP^1_R$ over a discrete valuation ring $R$,
this is essentially equivalent to locate the branch locus of
the generic fiber of $Y\to Z$. The language of clusters  (\cite{DDMM},
\S 1.1, \cite{Mu}) is then more explicit and convenient. 
The novelty here is that our method works in any (mixed or equal)
characteristic. It can also be applied to the case when $Z$ is a regular
model over $R$ of a smooth curve $Z_K$ over $K:=\Frac(R)$ and
$Y_K\to Z_K$ is a double cover of smooth curves over $K$. 
\medskip

Our initial motivation is to give an easily implementable algorithm to finding
 a desingularization of the surface $W$ in Example~\ref{exd}, in the
 spirit  of Tate's algorithm  for elliptic curves \cite{Tate}. Indeed,
 a  regular model of $C$ over 
 $\Z$ allows to compute some important arithmetic invariants of $C$ as
 its Artin conductor,  the Tamagawa number of its Jacobian $J(C)$ as
 well  as  the (finite part of) $L$-function of  $J(C)$.
 Applications will be given in \cite{LR}. An 
 implementation in PARI/gp \cite{pari} by B. Allombert is ongoing.
 \medskip

Now let us describe the content of this paper.
In \S~\ref{sect:lambda}, we introduce our multiplicity $\lambda_p(Y)$ for any
double cover $Y\to Z$ of regular scheme and for any point $p\in Y$. This invariant already
appeared in various works, when $Z$ is a smooth surface over a
perfect field, or in the context of Hironaka's characteristic
polyhedra (see {\it e.g.} \cite{CJSc}, \S 5). In our setting it is
particularly simple to define. 

Normalized blowing-ups are defined in \S~\ref{normal-bl} as 
blowing-ups followed by a normalization. These are the morphisms
appearing in the top line of Sequence~\eqref{eq:lipman}.  
We prove the main theorem~\ref{desing-main} about the description of 
a sequence of normalized blowing-ups. The equation of each $Y_i$ is 
given in Proposition~\ref{lemma2.12}.  

In \S~\ref{control-l}, we give a bound on the
sum of the multiplicities running through the closed points of the exceptional locus
of the normalized blowing-up $Y_1\to Y$ (Theorem~\ref{lambda-sum}).
Finally in \S~\ref{algo} we apply the above results to describe a
desingularization algorithm for the double covers of regular excellent
surfaces. 
\medskip

The general setting in this work is the following: $Z$ is an
\emph{integral noetherian excellent regular scheme}, $Y\to Z$ is a
\emph{double cover}.
No assumption are made on the residue fields of $Z$.  Except in \S~\ref{sect:lambda}, $Z$ will be of dimension $2$. 
\end{section}

\begin{section}{The multiplicity \texorpdfstring{$\lambda_p(Y)$}{lambdapY}} 
  \label{sect:lambda}
  
  The multiplicity of $Y$ at a closed point $p$ is a local invariant
  depending only on $\cO_{Y,y}$.   
  Let $A$ be an integral noetherian regular ring.
  Let $B=A[y]$ be an $A$-algebra with $y^2+ay+b=0$ for some  $a, b\in A$.    
We will define a multiplicity $\lambda_{\m}(B)\in \N $ for any maximal ideal
$\m$ of $A$. It is a measure of singularity of $B$ at the maximal
ideals lying over $\m$. 

\subsection{Definition and basic properties for local rings} 
Let us first recall some elementary facts on regular rings.
Let $\m$ be a maximal ideal of $A$,  let $d=\dim A_\m$ and
$k=A/\m$.  Consider  
$$\Gr_\m(A)=\oplus_{\ell\ge 0} \m^\ell/\m^{\ell+1}.$$
It is canonically a 
homogeneous $k$-algebra. As $A$ is regular, it is isomorphic to
$\mathrm{Sym}_k(\m/\m^2)$, the polynomial ring in $d$ variables over $k$. 

For any element non zero element $a\in A$, $\ord_\m(a)$, or simply 
$\ord(a)$ if $\m$ is implicitly given, is the supremum
of the integers $n$ such that $a\in \m^n$. 
It depends only on the ideal $a A$. By convention
$\ord(0)=+\infty$. We have $\ord(ab)=\ord(a)+\ord(b)$ for
$a, b\in A\setminus \{ 0 \}$.  
As $\Gr_\m(A)$ is an integral domain, $\ord$ extends to a 
valuation on $\Frac(A)$.

\begin{definition}\label{def-initf} Keep the above notation. If
  $\ord_\m(a)=\ell \ne +\infty$, the \emph{initial form of $a$} is the image of
  $a$ in $\m^{\ell}/\m^{\ell+1}$. The latter can be viewed as
  the $k$-vector space of homogeneous polynomials in $d$ variables
  and of degree $\ell$. 
\end{definition}

\begin{lemma}\label{CM-dc} Let $R$ be a
  noetherian ring.  Let $D$ be an $R$-algebra, free of rank $2$ as
  $R$-module. Then there exists $y\in D$ such that $D=R\oplus Ry$, and
  we have for some $a, b\in R$
  $$
y^2+ay+b=0. 
  $$
Moreover, any other $z$ such that $D=R[z]$ satisfies $z=uy+v$ with
  $u\in R^*$ and $v\in R$. 
\end{lemma}

\begin{proof} Locally on $\Spec R$, $R$ is a direct summand in 
  $D$ (\cite{Bou}, II, \S 3, Proposition 6 with $M=R$
  and $N=D$),   then apply {\it op. cit.}, Cor. 1 to Proposition 12
  with  $M=D$ and $N=R$, which implies that $R$ is a direct summand in
  $D$. The existence of the equation on $y$ and the statement on $z$ are obvious. 
\end{proof}

\begin{definition}\label{def-lambda} Let $A$ be an integral noetherian
  regular ring.  Let $B$ be a reduced $A$-algebra, 
free of rank $2$ as $A$-module. Write $B=A[y]$ with 
\[ y^2+ay+b=0 \]
for some $a, b\in A$. 

  Fix a maximal ideal $\m\subset A$.
For each generator $y$ as above, denote by  
$$\lambda_\m(y)=\min\{ 2\ord_\m(a), \ord_\m(b)\}.$$ 
Define 
$$\lambda_\m(B)=\sup_y \{ \lambda_\m(y)  \}\in \mathbb N \cup \{ +\infty\},$$
where the supremum runs through all generators $y$ as above. Note that
when an $y$ is given, we only need to compute $\lambda(y+c)$ for all $c\in
A$ to determine $\lambda_\m(B)$. 
\end{definition}

\begin{remark} The above multiplicity is  very likely twice the invariant $\delta(B)\in \mathbb Q_{\ge 0}$ considered by \cite{CJSc}, \S 5,  in a much more
general context involving Hironaka's characteristic polyhedra.
We use the notation $\lambda_\m(B)$ to keep coherent with our previous
work \cite{LTR}, D\'efinition 10 (with $r=1$).

The multiplicity $\lambda_\m(B)$ is local, in the sense that it depends only on
$A_\m \to B_\m$ as shown in the lemma below. 
\end{remark}

\begin{lemma} \label{lambda-local} Keep the above notation.
  \begin{enumerate}[\rm (1)] 
\item We have $\ord_\m=\ord_{\m A_\m}$ on $\Frac(A)$. 
\item  Consider the $A_\m$-algebra $B_\m$, then 
  $$\lambda_\m (B)=\lambda_{\m A_\m} (B_\m).$$
  \end{enumerate}
\end{lemma}

\begin{proof} (1) This is because $\m^n A_\m\cap A=\m^n$ for all $n\ge 0$. 

 (2) Clearly $\lambda_\m (B)\le \lambda_{\m A_\m} (B_\m)$.
 Let $z$ be a generator of $B_\m$ as $A_\m$-algebra.
 Let $B=A[y]$. 
  Then there exists $u\in A_\m^*$, $v\in A_\m$ such that  $y=uz+v$.
  Approximate $v$ by a sequence $v_n\in A$: $v=v_n+\alpha_n$
  with $\alpha_n\in  \m^n A_\m$. Then
  $B=A[y-v_n]$ and we have,  if $n$ is big enough,  
$$\lambda_{\m A_\m} (z)=\lambda_{\m A_\m} (uz)=\lambda_{\m A_\m}(y-v_n)=
\lambda_\m (y-v_n) \le \lambda_\m(B).$$
This proves the inverse inequality. 
\end{proof}

\begin{proposition}\label{lambda-f} Let $A, \m$ and $B$ be as in
  Definition~\ref{def-lambda}, let  $k=A/\m$.
Let   $\Delta_{B/A}=(a^2-4b)A$ be the discriminant ideal. 
  \begin{enumerate}[{\rm (1)}]
  \item We have $\ord_\m(\Delta_{B/A})=0$ if and only if $A\to B$ is
    \'etale at $\m$.
 \item We have $\lambda_\m(B)\le \ord_\m(\Delta_{B/A})$,
 with equality when $\chara(k)\ne 2$. 
\item The following properties are equivalent:
  \begin{enumerate}[\rm (a)] 
  \item  $\lambda_\m(B)=+\infty$,
  \item $\chara(A)=2$, and for any generator $y$ we have $a=0$ and $b$ is a square in the $\m$-adic
    completion $\hat{A}_\m$ of $A$.
  \item  $B\otimes_A \hat{A}_\m$ is non-reduced.  
  \end{enumerate}
In particular, if $A$ is excellent, then $\lambda_\m(B)$ is finite. 
  \end{enumerate}  
\end{proposition}

\begin{proof} (1) is clear.

  (2)   We have 
  $$\ord_\m(\Delta_{B/A})=\ord_\m(a^2-4b)\ge \min\{ \ord_\m(a^2), \ord_\m(b) \}
=\lambda_\m(y).$$
  Thus $\lambda_\m(B)\le \ord_\m(\Delta_{B/A})$.
  When $\chara(k)\ne 2$,  by Lemma~\ref{lambda-local} 
  we can suppose $2\in A^*$. Then $B=A[y+a/2]$ and 
$\ord_\m(\Delta_{B/A})=\lambda_\m(y+a/2)\le \lambda_\m(B)$ which implies the equality.  

  (3) We can suppose $A$ is local.

(a) $\Longrightarrow$ (b) Suppose that $\lambda_\m(B)=+\infty$. Then there exists a sequence 
$(c_n)_n$ in $A$ such that $\lambda_\m(y+c_n)\ge 2n$. Thus
  $$a-2c_n\in \m^n, \quad b-ac_n+c_n^2\in \m^{2n}.$$
  If $\chara(A)\ne 2$, then $c_n$ converges to some $\hat{c}\in \hat{A}_\m$ and we have
  $a=2\hat{c}$.   As $A\to \hat{A}_\m$ is faithfully flat, $a\in A\cap 2\hat{A}_\m=2A$. Thus
  $\hat{c}=c\in A$. Then $b=c^2$ and $(y-c)^2=0$, contradiction
  because $B$ is reduced by hypothesis. 
  Therefore $\chara(A)=2$. Then $a\in \cap_n \m^n=\{ 0\}$ and $b+c_n^2\in \m^{2n}$.  This implies that $c_n$ converges to some $\hat{c}\in \hat{A}_\m$ with $\hat{c}^2=b$.

  (b) $\Longrightarrow$ (c) is clear.  Let us prove (c)
  $\Longrightarrow$ (a).  By Lemma~\ref{non-reduced}(1) below,
    there exists  $\hat{\alpha}\in \hat{A}_\m$ 
    such that $(y+\hat{\alpha})^2=0$.
   If $(\alpha_n)_n$ is a sequence in $A$ convergent to $\hat{\alpha}$, then
   $\lambda_\m(y+\alpha_n)$ tends to $+\infty$.

Finally, if $A$ is excellent then so is $B$, hence $B\otimes_{A}
\hat{A}_\m$ is reduced. 
\end{proof}

\begin{lemma}\label{non-reduced} Let $R$ be an integral normal
  domain, let $D=R[T]/(T^2+aT+b)$.
  \begin{enumerate}[\rm (1)] 
  \item The ring $D$ is not integral if and only if 
  $T^2+aT+b$ is reducible in $R[T]$.   
  \item  The ring $D$ is non-reduced if and only if $T^2+aT+b=(T+c)^2$ for some $c\in A$. 
  \end{enumerate}
\end{lemma}

\begin{proof} The ``if'' part in both statements are obvious. Let us prove the
  converse. Let $K=\Frac(R)$. As $D$ is flat over $R$, $D\to D\otimes_R K$
  is injective. If $D$ is non-integral, then $D\otimes_R K$ is non-integral and we have 
  $T^2+aT+b=(T+c_1)(T+c_2)$ for some  $c_1, c_2\in K$. This implies that the
  $c_i$'s are integral over $R$, hence belong to $R$. The proof is similar for
  the non-reduced case. 
  \end{proof} 

  \begin{remark} Keep the notation of Proposition~\ref{lambda-f}. We
    have $\Delta_{B/A}=0$ if and only if $B$ is integral and
    $\Frac(B)/\Frac(A)$ is an inseparable extension.     
  \end{remark}
  
\begin{proposition}\label{reg-lambda} Keep the notation of
  Definition~\ref{def-lambda}.
  \begin{enumerate}[\rm (1)] 
  \item The Artinian ring $B/\m B$ is reduced if and only if 
    $\lambda_\m(B)=0$. If $k:=A/\m$ is perfect or of $\chara(k)\ne 
    2$, then this is equivalent to $A\to B$ \'etale above $\m$. 
  \item The ring $B$ is regular at all maximal ideals lying over $\m$ 
    if and only if $\lambda_\m(B)\le 1$.
   \item   If $\lambda_\m(B)\ge 1$, then there is only one maximal
     ideal $\n\subset B$ lying over $\m$, $A/\m\to B/\n$ is an
     isomorphism and we have $B_\n=B_\m$.
  \end{enumerate}
   \end{proposition} 

\begin{proof} By Lemma~\ref{lambda-local} we can suppose $A$ is local.
  Suppose $\chara(k)=2$ (otherwise the proof is easier). 
  Let $\lambda_\m(y)=\lambda_\m(B)$.

If $\lambda_\m(B)=0$ and $\ord_\m(a)=0$, then $A\to B$ is \'etale and
$B$ is regular.
Suppose $\lambda_\m(B)=0$ and $\ord_\m(a)>0$. Then 
$B/\m B \simeq k[T]/(T^2+\bar{b})$. If $T^2+\bar{b}\in k[T]$ is reducible,
there exists $\alpha\in A^*$ such that $b-\alpha^2 \in \m$.
This implies that $\lambda_\m(B)\geq \lambda_\m(y+\alpha)>0$.
Contradiction.  So  $B/\m B$ is a radicial quadratic extension of $k$   
and $B$ is regular. 
  
If $\lambda_\m(B)\geq 1$, then  $a, b\in \m$ and $B/\m B\simeq
k[T]/(T^2)$. This implies that $B$ is local with maximal ideal $\n$,
not \'etale over $A$, and we have $k\simeq B/\n$. 
The equality $B_\n=B_\m$ comes
from the fact that for any $u\in B\setminus \n$, the norm $N_{B/A}(u)\in
A\setminus \m$.

Finally when $\lambda_\m(B)\geq 1$, we have $y\in \n$ and the
relation $y^2+ay+b=0$ implies that $B$ is regular if and only if 
$b\notin \m^2$, equivalently, $\lambda_\m(B)=1$.
 \end{proof}

 \subsection{Algorithm for determining the
   multiplicity \texorpdfstring{$\lambda_p(Y)$}{lambdap(Y)}}

We now discuss how to find a generator $y$ of $B$ reaching 
$\lambda_\m(B)$.  We can restrict to  $\chara(k)=2$ by
Proposition~\ref{lambda-f}. 
The lemma below is a generalization of \cite{LRN}, Lemma 3.3(1-2).
It gives an algorithm to finding a generator  $y$ of $B$ with
$\lambda_\m(B)=\lambda_\m(y)$.

\begin{lemma} \label{lambdaz} Let $A, \m, B$ be as in
  Definition~\ref{def-lambda} with $\chara(A/\m)=2$. 
  We omit $\m$ from the subscripts of $\ord$ and $\lambda$.  
\begin{enumerate}[\rm (1)] 
\item If $2\ord(a)\le \ord(b)$, then $\lambda(B)=\lambda(y)$. 
\item Suppose $2\ord(a)> \ord(b)$. Let $\tilde{b}\in \Gr_\m(A)$ be
  the initial form of $b$ (see Definition~\ref{def-initf}). 
  \begin{enumerate}
\item If  $\tilde{b}$ is not a square in $\Gr_\m(A)$, then 
$\lambda(B)=\lambda(y)$.    
 \item Suppose $\tilde{b}$ is a square in $\Gr_\m(A)$: 
    there exists $c\in A$ such that $b-c^2\in \m^{\ell+1}$ where
    $\ell=\ord(b)$. Then 
    \[ \lambda(y+c)>\lambda(y).\] 
    \end{enumerate}
\end{enumerate} 
\end{lemma}

\begin{proof} Up to multiplication by a unit of $A$, any other basis
  $\{ 1, z\}$ of $B$ over $A$ is given by $y=z+c$ for some $c\in
  A$. The equation of $z$ is
  $$
z^2+(a+2c)z+(b+ac+c^2)=0.  
$$
Using the fact that $\ord$ is a valuation on $\Frac(A)$, the proof of
(1) is similar to that of Lemma 3.3(1) in \cite{LRN}.

(2.a) If $\lambda(B)>\lambda(y)$, there exists   
$c\in A$ such that
\[ \ord(b+ac+c^2)>\lambda(y)=\ord(b).\]
As
$2\ord(a)>\ord(b)$, this implies that $\ord(b)=2\ord(c)$ and
$\ord(b+c^2)>\ord(b)$. This is equivalent to the initial form of
$b$ being  a square in $\Gr_\m(A)$. 

(2.b) We have $\ord(b)=2\ord(c)$. Moreover
\[2\ord(a+2c)>\ord(b), \quad 
  \ord(b+ac+c^2)=\ord(b-c^2+(a+2c)c)> \ord(b). 
  \] 
  Hence $\lambda(y+c)>\ord(b)=\lambda(y)$.
\end{proof}

\begin{remark}\label{lambda-Delta} If $\chara(A/\m)=2$, and $\Delta_{B/A}\ne 0$,
  then $\lambda_\m(B) = \ord_\m(\Delta_{B/A})$ if and only if for some
  generator $y$ of $B$ we have $2\ord_\m(a)\le \ord_\m(b)$. 
This then implies that $\lambda_\m(B)=\lambda_\m(y)$.  
\end{remark}

\begin{example} Let $A=\Z[x]$. Let
  $$B=\Z[x,y]/(y^2+x^5-1).$$
and let $\m\subset A$ be a maximal ideal. 
\begin{enumerate}[{\rm (1)}] 
\item   As $B\otimes_\Z \mathbb Q$ is regular,  we have $\lambda_\m(B)\le 1$ 
  if $\m\cap \Z=\{ 0\}$. 
\item  Suppose now that $\m\cap \Z=\ell \Z$ for some prime number $\ell$.
There exists a monic $f(x)\in \Z[x]$ such that $f(x)$ is irreducible
in $\mathbb F_\ell[x]$ and $\m=(\ell, f(x))$.
\begin{enumerate}
\item 
If $\ell\ne 2, 5$, then $\lambda_\m(B)\le 1$, with equality if and
only if $f(x) \mid x^5-1$ in $\mathbb F_\ell[x]$. 
\item 
If $\ell = 5$, then $\lambda_\m(B) = 0$ if $f(x)$ is prime to $x-1$
in $\mathbb F_5[x]$.
Let $\m_{5}:=(5, x-1)$.
Then
$\lambda_{\m_5}(B)=\ord_{\m_5}(x^5-1)=2$. 
\item 
Let $\ell=2$. Let $z=y-1\in B$. Then $z^2+2z+x^5=0$.  Consider
$\m_2=(2, x)$. By Lemma~\ref{lambdaz}(1), we have
$\lambda_{\m_2}(B)=2\ord_{\m_2}(2)=2$. 
\item Let $\ell=2$ and $\m\ne \m_2$, then $\lambda_{\m}(B)=1$.
  Indeed,  the curve $z^2+x^5$ over $\mathbb F_2$ is smooth away
  from the point $x=z=0$, hence $\Spec B\to \Spec \Z$ is smooth at
  $\m$ and $B_\m$ is regular. Thus   $\lambda_{\m}(B)=1$ (it is
  non-zero by Proposition~\ref{reg-lambda}(2)).   
\end{enumerate}
\end{enumerate} 
\end{example}

\begin{proposition}\label{lambda-et} Let $(A, \m)\to (A', \m') $ be an
  \'etale extension of regular local rings.  Then
    $$\lambda_\m(B)=\lambda_{\m'}(B\otimes_A  A').$$      
\end{proposition}

\begin{proof} Let $k, k'$ be the respective residue fields of $A, A'$. 
  As $\Gr_\m(A)\to \Gr_{\m'}(A')\simeq \Gr_{\m}(A)\otimes_k k'$ is
  injective, $\ord_\m$ coincides with $\ord_{\m'}$ on $A$. 

Let $y\in B$ such that $\lambda_\m(y)=\lambda_\m(B)$. Then
$B'=A'[y]$.  If $\chara(k)\ne 2$ then $\lambda_{\m'}(B')= 
\ord_{\m'}(\Delta_{B/A}) =\ord_\m(\Delta_{B/A})$. Suppose
$\chara(k)=2$.
By Lemma~\ref{lambdaz},
$\lambda_\m(y)=\lambda_{\m'}(y)=\lambda_{\m'}(B')$ (in the case \ref{lambdaz}(2)
use the fact that if an element of $\Gr_{\m}(A)$ is a square in
$\Gr_{\m}(A)\otimes_k k'$, then it is already a square in $\Gr_{\m}(A)$ because $k'$ is separable over $k$.) 
\end{proof}

\subsection{Multiplicity in the global case}

\begin{definition} \label{double-cov} Let $Z$ be an integral
noetherian excellent regular scheme. Let $\psi : Y\to Z$  be a double cover
(Definition~\ref{def-dc}). 
  Then for any $p\in Y$ with image $q\in Z$, define 
$$
\lambda_p(Y)=\lambda_q(Y)=\lambda_{\m_q}(\psi_*\cO_{Y}\otimes_{\cO_Z} \cO_{Z,q}). 
$$
We define the \emph{discriminant ideal $\Delta_{Y/Z}\subset
  \cO_Z$} as the norm $N_{\psi_*\cO_Y/\cO_Z}$ of the annihilator
ideal of $\Omega^1_{Y/Z}$. It coincides with the usual definition when
$Z=\Spec A$ is affine. 
\end{definition}

\begin{remark}
The multiplicity $\lambda_p(Y)$ is the same for all points $p\in
\psi^{-1}(q)$, so there is no ambiguity to denote it by
  $\lambda_q(Y)$. Moreover, when $\lambda_p(Y)\ge 1$, then
  $\psi^{-1}(q)=\{ p\}$ and $k(q)\to k(p)$ is an isomorphism
  (Proposition~\ref{reg-lambda}(3)).
\end{remark}

\end{section}

\begin{section}{Normalized blowing-up} \label{normal-bl}

For any scheme $Y$ and closed subscheme $Y_0\subset Y$,
we call \emph{the normalized  blowing-up of $Y$ along $Y_0$}
the normalization of the blowing-up of $Y$ along $Y_0$.  
In this section we study the normalized blowing-up of $Y$ at a
singular point when $Y$ is a \emph{normal double cover} $\psi: Y\to Z$ of
an integral noetherian excellent regular surface $Z$.

\subsection{Blowing-up regular schemes along a point} \label{bl-up-reg}

Let $A$ be a ring and let  $\m\subset A$ be a maximal ideal.  We recall below 
the explicit description of the blowing-up 
  $$\bl_\m : Z'\to Z=\Spec A$$
of $Z$ along $\m$. 
Denote by $\widetilde{A}$ the Rees algebra $\oplus_{n\ge 0} \m^n$. For any $s\in \m\setminus \{ 0\}$, denote by $\tilde{s}$ the element $s$ viewed as a homogeneous element of
degree $1$ in $\widetilde{A}$ and by $\Ast$ the homogeneous localization of
$\widetilde{A}$ with respect to the positive powers of $\tilde{s}$. 
Denote by $D_+(\tilde{s})\subseteq Z'$ the principal open subset
defined by $\tilde{s}$. 

\begin{proposition}\label{bl-reg-ufd} Let $A$ be an integral
  noetherian regular ring, let $\m\subset A$ be a maximal ideal 
  of $\dim A_\m=2$. Let $s, t$ be a system of generators of $\m$ and
  denote $k=A/\m$. 
  \begin{enumerate}[\rm (1)] 
  \item We have
    \[ Z'\simeq \Proj A[S, T]/(tS-sT),\] 
$Z'$ is a regular surface, covered by  
    the principal open subsets $D_+(\tilde{s})$ and $D_+(\tilde{t})$.    
    The homomorphism 
\[ A[T_1]/(sT_1-t) \to \Ast=\cO_{Z'}(D_+(\tilde{s})), \quad T_1\mapsto
\tilde{t}/\tilde{s}\] 
of $A$-algebras is an isomorphism, it takes $\m^r\subseteq A$ onto
$s^r\Ast$. 
\item  Let  $E=V(\m \cO_{Z'})\subset Z'$ be the exceptional divisor.
Then
\[\cap D_+(\tilde{s})=V(s\Ast)\simeq \Spec k[\bar{t}_1]\] 
where $\bar{t}_1$ is the images of
$t_1:=\tilde{t}/\tilde{s}$ modulo $s$. 
 Any  maximal ideal of $\Ast$ containing $s$ is generated by $s$ and
 one other element.
 \item  Let $\nu_E$ be the normalized discrete valuation on $\Frac(A)$
associated to the generic point  
of $E$.  Then $\nu_E(\alpha)=\ord_\m(\alpha)$ for all $\alpha\in \Frac(A)$.   
 \item If $A$ is factorial  then so is $\Ast$. 
\item Let $f\in A$ with $\ord_\m(f)=n\ge 1$. 
  \begin{enumerate}
    \item Let $V(f)'$ be the strict transform of $V(f)$ in $Z'$.   
      Then  we have 
  $V(f)'\cap D_+(\tilde{s})=V(f/s^n)$.  
\item The image of $f/s^n \in \Ast$ in
  $\Ast/(s)\simeq k[\bar{t}_1]$ has degree at most equal to $n$.
  \end{enumerate}
  \end{enumerate}
\end{proposition}

\begin{proof} (1), (2), (5.a) are standard.
  
  (3) Let $\xi_E$ be the generic point of $E$. By (2), $\xi_E \in D_+(\tilde{s})$ and $\cO_{Z', \xi_E}(-E)$ is generated by
  $s$. Therefore $\nu_E(s)=1$. If $\alpha\in \m^\ell\setminus \m^{\ell+1}$, then
  $\alpha/s^\ell\in \Ast$, hence $\nu_E(\alpha)\ge \ell=\ord_\m(\alpha)$.  

  Conversely, suppose $\nu_E(\alpha)=\ell$. We have 
  $$\mathrm{div}(\alpha s^{-\ell}) 
  =\mathrm{div}(\alpha)-\ell E\ge 0$$
as Weil divisors on   $D_+(\tilde{s})$.
Thus $\alpha\in s^\ell \widetilde{A}_{(\tilde{s})}$.
 So there exists $N\ge \ell$ such that $\alpha s^{N-\ell} \in \m^N$.
 Because the class of $s$ in
 $\m/\m^2\subset \Gr_\m(A)$ is a regular element,
  $\alpha\in \m^{\ell}$, hence $\ord_\m(\alpha)\ge \nu_E(\alpha)$. 

  (4) Notice that
  $s\in \Ast$ is a prime element by (2), and $(\Ast)_s=A_s$ is factorial.
  Hence $\Ast$ is factorial by Nagata's criterion. 

  (5.b) Write $$
f=\sum_{0\le i\le n} \alpha_i t^i s^{n-i}+ \epsilon_{n+1} 
$$
with $\alpha_i \in A$ and $\epsilon_{n+1}\in \m^{n+1}$. This implies that
$$
f/s^n=\sum_{0\in i\le n} \alpha_i t_1^i+ s c,  \quad c\in \Ast. 
$$ 
Its class mod $s$ has degree $\le n$. 
\end{proof}    

\subsection{Normalization after blowing-up} 

To determine the integral closure of an integral domain we will use Serre's
(S$_2$) + (R$_1$) criterion (\cite{LB}, Theorem 8.2.23). 

\begin{lemma} \label{normal-codim1} Let $(R, \nu)$ be a discrete
  valuation ring with residue field $\kappa $, let
  \[ D=R[y]/(y^2+\alpha y+\beta) \] 
  be an $R$-algebra with $\Delta_{D/R}:=\alpha^2-4\beta\ne 0$.  
\begin{enumerate}[\rm (1)]
\item If $\nu(\alpha^2-4\beta)=0$, then $D$ is an \'etale $R$-algebra of
  rank $2$;
\item If $\nu(\alpha^2-4\beta)=1$, then $D$ is a discrete valuation ring
  totally ramified over $R$ with ramification index $2$;
\item The same is true if $\chara(\kappa)=2$, $\nu(\alpha)>0$,  and
  $\nu(\beta)=1$; 
\item If $\chara(\kappa)=2$, $\nu(\alpha)>0$,  and 
  $\beta$ is not a square in $\kappa$, then $D$ is  a discrete valuation
  ring with ramification index $1$ over $R$ 
  and inseparable quadratic residue extension. 
\end{enumerate} 
\end{lemma} 

\begin{proof} Straightforward. 
\end{proof}

\begin{proposition} \label{lemma2.12} Keep the notation of 
 Proposition~\ref{bl-reg-ufd} with $A$ excellent. Let $B$ be  
  an integrally closed domain, free of rank $2$ over $A$. 
  Fix a generator $y$ such that $\lambda_\m(y)=\lambda_\m(B)$ and 
  set $r=[\lambda_\m(B)/2]$.   
  Let $s\in \m \setminus \m^2$ and let $B_1:=\Ast[y_1]\subset \Frac(B)$.
  Put 
\[ y_1=y/s^{r}, \ a_1=a/s^{r}, \ b_1=b/s^{2r}\in A_s. \] 
Then
\begin{enumerate}[\rm (1)] 
\item $\Delta_{B_1/\Ast}=s^{-2r}\Delta_{B/A}\Ast$; 
\item $a_1, b_1\in \Ast$,  
    \begin{equation}
      \label{eq:z1}
y_1^2+a_1y_1+b_1=0       
    \end{equation}
    and $B_1$ is the integral closure of $\Ast$ in $\Frac(B)$.
\end{enumerate} 
\end{proposition}

\begin{proof} The hypothesis $A$ excellent insures that $\lambda_\m(B)$ is
  finite (Proposition~\ref{lambda-f}). Part (2) is immediate.

(1) We have $a_1, b_1\in \widetilde{A}_{(\tilde{s})}$
 by Proposition~\ref{bl-reg-ufd}(1). As $B_1$ contains $\Ast$
 and is finite over $\Ast$,  it is enough to show that $B_1$ is integrally closed. 
As $\Ast$ is regular and 
\[ B_1\simeq \widetilde{A}_{(\tilde{s})}[T]/(T^2+a_1T+b_1),\] 
$B_1$ is a local complete intersection, hence satisfies (S$_n$)
for all $n\geq 1$ (\cite{LB}, Corollary~8.2.18 and Example 8.2.20). By
Serre's criterion it is enough to show that $B_1$ is integrally closed at
points of codimension $1$. Moreover, as $\Spec B_1\to \Spec B$ is an
isomorphism away from $V(s)$, it is enough to check the property
for codimension $1$ points belonging to $V(s)$, so lying over the
generic point $\xi_E$ of $E$ (exceptional divisor of $Z'\to Z$).

Let $R=\cO_{Z', \xi_E}$. Then 
\[ B_1\otimes_{\Ast} R\simeq R[y_1]/(y_1^2+a_1y_1+b_1).\]
We will apply Lemma~\ref{normal-codim1}.  
We have $\min\{ 2\nu_E(a_1), \nu_E(b_1) \}\le 1$. So we only have to check that 
when $\chara(k(\xi_E))=2$, $\nu_E(a_1)>0$ and
$\nu_E(b_1)=0$, then $b_1$ is not a square in $k(\xi_E)$. Under these
hypothesis, we have $2\ord_\m(a)=2\nu_E(a) > 2r =
\nu_E(b)=\ord_\m(b)$ (Proposition~\ref{bl-reg-ufd}(3)).  

Suppose that $b_1$ is a square in $k(\xi_E)=\Frac(\Ast/(s))$. As
$\Ast/(s)$ is integrally closed (Proposition~\ref{bl-reg-ufd}(2)), $b_1$
is a square in $\Ast/(s)$: there exist $c_1, c_2\in \Ast$ such that
$$
b_1=c_1^2+sc_2. 
$$
There exist $N\ge r$ big enough, $\alpha_N\in \m^N$ and
$\beta_{2N}\in \m^{2N}$ such that $c_1=\alpha_N/s^N$,
$c_2=\beta_{2N}/s^{2N}$. Then
$$s^{2N-2r}b=\alpha_N^2+s\beta_{2N}.$$ 
So the initial form of $s^{2N-2r}b\in \m^{2N}$ is a square in $\Gr_\m(A)$,
thus the same is true for that of $b$. 
This contradicts Lemma~\ref{lambdaz}(2.b) and (1) is
proved. 
\end{proof} 

\subsection{Lipman's resolution of singularities} \label{SNB}

\begin{theorem} \label{desing-main} Let $Z$ be an integral noetherian
  excellent regular surface, 
  let $\psi: Y\to Z$ be a normal double cover. 
  Let $p_0\in Y$ be a singular point, let $q_0=\psi(p_0)$. 
  \begin{enumerate}[\rm (1)] 
  \item Let $Z'$ (resp. $Y'$) be the blowing-up of $Z$ (resp. of $Y$)
    along the (reduced) closed point $q_0$ (resp. $p_0$) and let
    $\tilde{Y}\to Y'$ be the normalization. Then 
we have a canonical commutative diagram 
    $$
\xymatrix{ \tilde{Y} \ar[r] \ar_\phi[rd] &   Y' \ar[r]^{{\mathrm{bl}}_{p_0}} \ar[d] & Y \ar^{\psi}[d]\\
  & Z' \ar^{{\mathrm{bl}}_{q_0}}[r] & Z 
}
$$
and $\phi: \tilde{Y}\to Z'$ is a double cover with discriminant ideal 
\[\Delta_{\tilde{Y}/Z'}=\cO_{Z'}(2rE)\otimes_{\cO_{Z}}\Delta_{Y/Z},\]
where $r=[\lambda_{q_0}(Y)/2]$ and $E$ is the exceptional divisor of $Z'\to
Z$.
\item If
  \[Y_n \to Y_{n-1} \to ... \to Y_1\to Y_0=Y\]
  is a sequence of normalized
    blowing-ups of (reduced) singular points $p_i\in Y_i$, then we have 
   a commutative  diagram  
\[
\xymatrix{
  Y_n \ar[r] \ar[d]^{\psi_n} &  Y_{n-1} \ar[r] \ar[d]^{\psi_{n-1}}& \cdots  \ar[r] &  Y_1
  \ar[d]^{\psi_1} \ar[r] &  Y_0 \ar[d]^{\psi_0=\psi} \\ 
   Z_n \ar[r]         &  Z_{n-1} \ar[r]          & \cdots  \ar[r] & Z_1
   \ar[r] &  Z_0:=Z 
 }
\] 
where $Z_{i+1}\to Z_i$ is the blowing-up of the reduced closed point
$q_i\in Z_i$, image of $p_i$, and each $\psi_i: Y_i\to Z_i$ is a
normal double cover. 
\item Suppose moreover that $Y\to Z$ is generically \'etale.
Then the Galois group $G:=\mathrm{Gal}(K(Y)/K(Z))$
acts on each $Y_i$ and  $Y_i/G=Z_i$.
  \end{enumerate} 
\end{theorem}

\begin{proof} (1) First we prove that $Y'\to Z$ factorizes through a
  finite morphism $Y'\to Z'$. We can suppose that $B$ is free over $A$.
  Let $\m\subset A$ be the maximal ideal corresponding to $q_0$. 
  Let  $\lambda_\m(y)=\lambda_\m(B)$. Then
  $a\in \m$ and $b\in \m^2$.   Let $\n\subset B$ be the maximal ideal
  corresponding to $p_0$. It is easy to that
  $$\n^n=\m^n\oplus\m^{n-1}y=\m^{n-1}\n$$
  for all $n\ge 1$. Therefore 
  $$\oplus_{n\ge 0} \n^n=Ay+
  (\oplus_{n\ge 0} \m^n)+ (\oplus_{n\ge 0} \m^n)\tilde{y}$$ 
where $\tilde{y}$ is $y$ viewed as homogeneous element of degree
$1$. Hence $\oplus_{n\ge 0} \n^n$ is finite over
$\oplus_{n\ge 0} \m^n$ (see also \cite{SH}, Proposition 5.2.5),
and $Y\to Z$ induces a finite morphism 
$Y'\to Z'$. Thus $\tilde{Y}$ coincides with the normalization of $Z'$
in $K(Y)$. The local structure of $\tilde{Y}$ is described by
Proposition~\ref{lemma2.12}, and we see that 
$\tilde{Y}\to Z'$ is a normal double cover (this also follows from
Example~\ref{exd_2}), and the discriminant  
ideal is $\cO_{Z'}(2rE)\otimes_{\cO_Z}\Delta_{Y/Z}$. 

(2) and (3) follow from (1). 
\end{proof}

\begin{remark} \label{bl-nor-w} Suppose
  $Z=\Spec A$ is affine and $Y=\Spec B$ is free over $A$.
  Consider the graded $A$-algebra $\tilde{A}[W]$
  with $\deg W=r$. We see using Proposition~\ref{lemma2.12}
  that $\tilde{Y}$ is the
  closed subscheme of $\Proj (\tilde{A}[W])$ defined by the equation
  $$
  W^2+\tilde{a}W+\tilde{b}=0,
  $$
where $\tilde{a}, \tilde{b}$ are respectively $a$ and $b$
considered as homogeneous elements in $\tilde{A}$ of
respective degrees $r, 2r$.
\end{remark}

 \subsection{The exceptional locus of Lipman's resolution of  singularities}

 Keep the assumption and notation of  Theorem~\ref{desing-main}. Suppose
 that Lipman's sequence for the singular point $p_0\in Y_0$ terminates at $Y_N$.
 The morphism $Y_N\to Y_0$ is an isomorphism away from $p_0$. 
 In this subsection we study the scheme structure of the pre-image of $p_0$ by
 $Y_N\to Y_0$ (the exceptional locus of $Y_N\to Y_0$). 
   
 We fix a prime divisor $E_0\ni q_0$ of $Z_0$, regular at $q_0$.
 Typically, for our applications in mind, $Z$ is a relative curve over
 a  discrete valuation ring $R$, and $E_0$ is the closed fiber of $Z$.
 Let
\[ \cE:=Y_N\times_{Y_0} \psi_0^*E_0\subset Y_N.\]
This is the union of the strict
 transform of $\psi^*E_0$ in $Y_N$ and the exceptional locus of $Y_N\to Y_0$. 
 In the above example, $\cE$ is the closed fiber of $Y_N$.  We would
 like to know the 
 structure of each irreducible component of $\cE$, their multiplicities in $\cE$, 
and how they intersect each other. 

The double cover $Y_0\to Z_0$ induces a finite flat morphism
$\cE\to Z_N\times_{Z_0} E_0$ 
of degree $2$. We start by describing $Z_N\times_{Z_0} E_0$. 
For all $1\le i\le N$, denote by $E_i\subset Z_i$ the
 exceptional divisor of $Z_i \to Z_{i-1}$ and by $E'_i$ (including $i=0$) its strict
 transform in $Z_N$.

\begin{lemma} \label{ZN_Z}
 Write
 \[ Z_N\times_Z E_0= \sum_{0\leq i \leq N} m_i E'_i\]
 as Cartier divisors of $Z_N$. 
\begin{enumerate}[\rm (1)] 
\item For $1\leq i\leq N$, $E_i=\PP^1_{k(q_{i-1})}$.
  For all $0\leq i\le N$, the canonical morphism $E'_i\to E_i$ is an
  isomorphism. Denote by $q'_i\in E'_i$ the pre-image of $q_i$.  
\item The union $\cup_{0\leq i\leq N} E'_i\subset Z_N$ is a tree of irreducible
  components with transverse intersections. 
\item The multiplicities $m_i$ can be determined inductively as
  follows: $m_0=m_1=1$. Let $i\geq 1$. If $q_{i}$ is the intersection 
point of $E_i$ with the strict transform of $E_{i-1}$ in $Z_i$, then 
    $m_{i+1}=m_{i}+m_{i-1}$. Otherwise, $m_{i+1}=m_i$. 
  \end{enumerate}
\end{lemma} 

\begin{proof} (1) This is because $E'_i\to E_i$ is finite birational
  and $E_i$ is normal.

  (2)-(3) Apply inductively Proposition~\ref{bl-reg-ufd} to the $Z_i$'s. One
 can take for $s$ a generator of $\cO_{Z}(-E_0)_{q_0}$. 
\end{proof}

\begin{proposition} \label{birational} Keep the notation of Lemma~\ref{ZN_Z}.
For $1\leq i\leq N$, set
 \[ \Theta'_i=\psi_N^{-1}(E'_i) \]   with the reduced structure. 
 \begin{enumerate}[\rm (1)]
 \item The scheme $\Theta'_i$ is irreducible if and only if $\psi_i^{-1}(E_i)$ is
   irreducible.  When it is reducible, it has two irreducible components, 
      each of them is isomorphic to $E'_i$ by $\psi_N$, and 
has multiplicity $m_i$ in $\cE$. 
 \item Suppose that $\lambda_{q_{i-1}}(Y_{i-1})$ is odd,  then $\Theta'_i\to E'_i$
    is an isomorphism, and the multiplicity of $\Theta'_i$ in $\cE$ is $2m_i$.
  \item Suppose that $\lambda_{q_{i-1}}(Y_{i-1})$ is even and
    $\Theta'_i$ is irreducible. Then  the latter has multiplicity $m_i$ in $\cE$.
Moreover, if $\chara(k(q_0))\neq 2$, then $\Theta'_i\to E'_i$ is generically \'etale. 
  \end{enumerate}
 \end{proposition}

\begin{proof} Note that any finite birational morphism onto $E'_i$ is an
  isomorphism because the latter is normal. So it is enough to
  study $\Theta'_i$ at its generic points. Note that 
the multiplicity in $\cE$ of an irreducible component of $\Theta'_i$  
  is equal to $m_i$ multiplied by the ramification index over   $E'_i$. 

  As $Y_N\to Y_i$ is an isomorphism
  over an open neighborhood of the generic point of $E_i$, it is enough to
  show the statement for $\Theta_i:=\psi_i^{-1}(E_i)$. Moreover 
  without loss of generalities we can assume that $i=1$.
  Then all the statements of the proposition follow from direct observations
  on an affine equation~\eqref{eq:z1} of $Y_1$ 
\end{proof}

\begin{remark} Keep the notation of Propositions~\ref{birational} and 
  \ref{lemma2.12}.
  \begin{enumerate}[\rm (1)] 
    \item The scheme $\psi_1^{-1}(E_1)$ is irreducible if and only if
  the polynomial
  \[ T^2+\bar{a}_1T+\bar{b}_1\in k[\bar{t}_1][T]\]
  (where $\bar{*}$ means the class modulo $s$)  is irreducible.
  The same statement applies
  to $\psi_i^{-1}(E_i)$ if a similar equation of $Y_i$ is known. 
\item Suppose $\chara(k(q_0))=2$ and $\Theta_1$ is irreducible.
  The $\Theta'_1\to E'_1$ is generically \'etale if and only if
  $2\ord_\m(a)=\lambda_{q_0}(Y)$.
\item The structure of $\Theta'_1$, in the case (3) of Proposition~\ref{birational}
  is more complicate to describe. We have to determine inductively the
  strict transform of $\Theta_1$ after each blowing-up.   
  \end{enumerate}
\end{remark}

\begin{proposition} \label{inter} Keep the notation of Lemma~\ref{ZN_Z} and 
  Proposition~\ref{birational}. 
Let $0\leq i\neq j\leq N$ be such that $E'_i\cap E'_j=\{ q \}\neq\emptyset$.  
 \begin{enumerate}[\rm (1)] 
      \item If $\Theta'_i$ has two irreducible components
     $\Theta'_{i,1}, \Theta'_{i,2}$  and similarly for $\Theta'_j$, then
     up to renumbering, 
     $\Theta'_{i,1}$ and $\Theta'_{j,1}$ (resp. $\Theta'_{i,2}$ and $\Theta'_{j,2}$) 
     intersect transversely at a single $k(q)$-rational point, 
     and $\Theta'_{i, 1}\cap \Theta'_{j,2}=\Theta'_{i, 2}\cap \Theta'_{j,1}
     =\emptyset$.
   \item If $\Theta'_i$ has two irreducible components  and $\Theta'_j$ is
     irreducible,  
     then $\Theta'_j$ intersects transversely each of these components at a single
     $k(q)$-rational point. These intersection points can be the same. 
   \item Suppose that $\Theta'_i$ and $\Theta'_j$ are both irreducible. Denote by
     $e_i, e_j$ the respective ramification indexes of $\psi_N$ at the
     generic points of $\Theta'_i, \Theta'_j$. Then we have 
     \[
e_{i}[k(\Theta'_i): k(E'_i)]=e_j[k(\Theta'_j): k(E'_j)] =2 
   \]
   and
   \[
   \sum_{p\in\psi_N^{-1}(q)} i_p(\Theta'_i, \Theta'_j)e_{i}e_j[k(p):k(q)]=2
     \]
where $i_p(\Theta'_i, \Theta'_j)$ denotes the intersection multiplicity of
  $\Theta'_i$ and $\Theta'_j$ at $p$. 
 \end{enumerate}
 \end{proposition}

\begin{proof} Let $\{ \Theta'_{i, \alpha}\}_\alpha, \{ \Theta'_{j, \beta} \}_{\beta}$
 be the respective irreducible components of $\Theta'_i$ and $\Theta'_j$.
Let $e_{i,\alpha}$ be the ramification index of $\psi_N$
    at the generic point of $\Theta'_{i, \alpha}$, similarly for $e_{j, \beta}$. We have 
\[
\sum_{\alpha} e_{i,\alpha}[k(\Theta'_{i,\alpha}): k(E'_i)]=[K(Y): K(Z)]=2. 
\]
Similarly for $\Theta'_j$. 
By general intersection theory, we have  
\[ \sum_{p\in\psi_N^{-1}(q)} i_p(\psi_N^*E'_i, \psi_N^*E'_j)[k(p):k(q)]
=[K(Y): K(Z)]i_{q}(E'_i, E'_j)=2, 
\] 
and by definition $\psi_N^*E'_i=\sum_{\alpha} e_{i, \alpha} \Theta'_{i,\alpha}$
(similarly for $\psi_N^*E'_j$). 
Therefore
\[
\sum_{p\in\psi_N^{-1}(q)} \sum_{\alpha, \beta} i_p(\Theta'_{i,\alpha}, \Theta'_{j, \beta})
e_{i, \alpha}e_{j, \beta}  [k(p):k(q)]=2. 
\] 
All coefficients in the above equality are positive except the intersection
multiplicities $i_p$ which can be zero. This leaves very
few possibilities for the combinations of the various coefficients. 

We also notice that through any point of $\psi_N^{-1}(q)\cap \Theta'_i$ passes
a point of $\psi_N^{-1}(q)\cap \Theta'_j$, by  the going down property of
$Y_N\to Z_N$ (see \cite{SP}, tag 00HU). The proposition is then proved by
straightforward case-by-case analysis.  
 \end{proof}

\subsection{A complete example of resolution of singularities}

\begin{example} \label{ex2} 
  Let $Z=\PP^1_{\Z_2}$ be the projective line over the $2$-adic integers.
  Fix two integers $n, m$ such that at least one of them is non-zero.
  Let $Y$ be the projective scheme over $\Z_2$ associated to  the affine equation  
  $$
y^2=(s^2-2)^3+16ms^2+n^2. 
  $$
Replacing $y$ with $y+s^3+n$, we get 
\begin{equation}
  \label{eq:init}
  y^2+2(s^3+n)y+2(3s^4 +ns^3 +2m_1s^2 + 4)=0
\end{equation}
where $m_1:=-(4m + 3)$. 
This is a normal Weierstrass model of  a genus 2 curve 
(\cite{LRN}, Lemma 2.3.1(c)).
The closed fiber of $Y$ over $\Spec \Z_2$ is a double (multiplicity
$2$) projective line over $\mathbb F_2$. Denote it by  $2\Theta_0$.
The aim of this example is to determine
the minimal resolution of singularities of $Y$ 
when $8 \nmid n$. 
Denote by $\boxed{t=2}$ the uniformizing element of $\Z_2$.

For each singular point of $Y$, the notations $Y_i$ and $Z_i$ refer 
to the blowing-ups described by Theorem~\ref{desing-main}. We
use the same notation for different singular points for simplicity.
When we describe the equations of $Y_i$,
$D_+(\tilde{f})$ always means open subset in $Z_{i+1}$
(for instance, we frequently consider  $D_+(\tilde{s})$ as open
subset of $Z_i$ for $i>1$). 
\medskip

(I) We first deal with the point $q_\infty$ corresponding to $1/s=0$.
Let $x=1/s$ and $z=y/s^3$, we get:  
\begin{equation} \label{eq:eq5} 
  z^2+2(1+nx^3)z+2x^2(3 + nx +2m_1x^2 +4x^4)=0.
\end{equation} 
with $\lambda_{q_\infty}(Y)=2=\lambda_{q_\infty}(z)$.  

(I.1) {\it First  blowing-up}. We blow-up $Y_1\to Y$ along
the point lying over $q_\infty$.
Over $D_+(\tilde{x})$, we have  $t=xt_1$ and 
\begin{equation} \label{eq:eq6}
z_1^2+t_1(1+nx^3)z_1+t_1x(3+nx+2m_1x^2+ 4x^4)=0. 
\end{equation}
There is a unique singular point $p_{\infty, 1}=(t_1=x=z_1=0)$ with
multiplicity $\lambda_{p_{\infty,1}}(Y_1)=2$. 

Over $D_+(\tilde{t})$, the equation is ($x_1=x/t$) 
$$
w_1^2+(1+nt^3x_1^3)w_1+tx_1^2(3+ntx_1+2m_1t^2x_1^2+4t^4x_1^4)=0
$$
which defines two regular points at $x_1=0$ (poles of $t_1$).

The exceptional locus $\Theta_{ 1}$ (with the reduced structure) of
$Y_1\to Y$ is union of two projective lines over $\F_2$, of multiplicity $1$,
intersecting at $p_{\infty,1}$.  
The strict transform of $\Theta_0$ intersects $\Theta_1$ at $t_1=0$
which is $p_{\infty, 1}$. 
\smallskip

(I.2)  {\it Second blowing-up. } We blow-up $Y_2\to Y_1$ along $p_{\infty, 1}$.
We use Equation~\eqref{eq:eq6}.   
Over $D_+(\tilde{x})\subset Z_1$ we have (letting $t_1=xt_2$) 
$$
z_2^2+t_2(1+nx^3)z_2+t_2(3+nx+2m_1x^2 + 4x^4)=0. 
$$
Over $D_+(\tilde{t}_1)$, we have $x=t_1x_1$ and 
$$
w_2^2+(1+nt_1^3x_1^3)w_2+x_1(3+nt_1x_1+2m_1t_1^2x_1^2 + 4t_1^4x_1^4)=0. 
$$
We see that the exceptional locus is a smooth conic
of multiplicity $2$ (we have $t=x^2t_2$). 
Its intersection with the strict transform of $\Theta_0$ is the point $t_2=z_2=0$.  
The intersection with the strict transform of $\Theta_1$
are the points $x_1=0,  w_2= 0, 1$.    

The exceptional locus of $Y_1\to Y$ and $Y_2\to Y_1$ as well as the strict
transforms of $\Theta_0$ 
are as shown in the figure below. In $Y_2$, the components are isomorphic to
$\PP^1_{\mathbb F_2}$, the intersections are transverse at rational points. 
  \begin{center} 
  \includegraphics[scale=0.45]{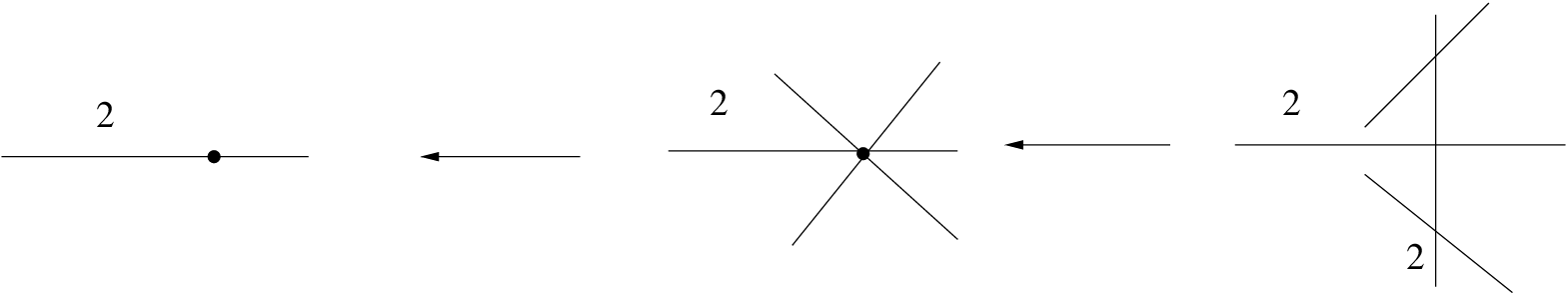}
  \end{center}  

The numbers indicate the multiplicities of the components. 
\medskip 

(II) Now come back to Equation~\eqref{eq:init}. The singularities depend on
the parity of $n$.  Indeed, when $n$ is even, there is a unique singular
point $p_0$, lying  above $q_0=(s=2=0)$ with $\lambda_{q_0}(Y)=3$. 
When $n$ is odd, there are two singular points: $p_0$ with $\lambda_{q_0}(Y)=2$
and $p_{*}$ lying over $q_{*}=(s+1=2=0)$, with 
$\lambda_{q_{*}}(Y)=2$.  

We start with the simpler case $\boxed{n \text{ is odd}}$.

(II.1) {\it First blowing-up.} We blow-up $Y_1\to Y$ along $p_0$. We
use Equation~\eqref{eq:init}.  
The equation of $Y_1$ over $D_+(\tilde{s})$  is ($2=t=st_1$): 
\begin{equation} \label{eq:eq7} 
y_1^2+t_1(s^3+n)y_1+st_1(3s^2 +ns +2m_1st_1 + t_1^2)=0.
\end{equation} 
There is a unique singular point $p_1$, lying 
over $q_1=(s=t_1=0)$, with $\lambda_{q_1}(Y_1)=2$.
Over $D_+(\tilde{t})$, we have $s=s_1t$ and
$$w_1^2+(s_1^3t^3+n)w_1+t(3s_1^4t^2 +ns_1^3t +2m_1s_1^2 + 1)=0$$ 
which is regular at $s_1=0$ (two points, poles of $t_1$). 

The exceptional locus $\Theta_1$ of $Y_1\to Y$ above $p_0$ is
a reduced degenerate conic as $n$ is odd. 
It meets the strict transform of $\Theta_0$ at  $p_1$.   
\smallskip

(II.2) {\it Second blowing-up. } We blow-up $Y_2\to Y_1$ along $p_1$.
We use Equation~\eqref{eq:eq7}. 
Over $D_+(\tilde{s})\subset Z_2$, we have $t_1=st_2$ and 
\begin{equation} \label{eq:eq81}
  y_2^2+t_2(s^3+n)y_2+t_2s(n+3s +m_1st_2 + st_2^2)=0.
\end{equation} 
Over $D_+(\tilde{t}_1)$, the equation of $Y_2$ is ($s=s_2t_1$) 
$$
w_2^2+(n+s_2^3t_1^3)w_2+s_2t_1(3s_2^2t_1 +ns_2+m_1s_2t_1+ t_1)=0, 
$$
which is regular at $s_2=0$. 

The exceptional locus $\Theta_2$ is union of two projective lines over $\F_2$,
of multiplicity $2$ and intersecting at $p_2=(t_2=y_2=s=0)$ with
$\lambda_{p_2}(Y_2)=2$. They have multiplicity $2$ because $t=s^2t_2$. The intersection with the strict transform of $\Theta_0$  is $p_2$, 
and with that of $\Theta_1$ is $s_2=0$ (two points). 
\smallskip
 
(II.3) \emph{Third blowing-up}. We blow-up $Y_2$ along the unique singular point
$p_2$ (lying over $p_0$). 
Over $D_+(\tilde{s})\subset Z_3$, we have ($t_2=st_3$) 
$$
y_3^2+t_3(s^3+n)y_3+t_3(n+3s +m_1s^3t_3 + s^3t_3^2)=0.
$$
Over $D_+(\tilde{t}_2)$, we have ($s=s_3t_2$) 
$$
 w_3^2+(s_3^3t_2^3+n)w_3+s_3(n+3s_3t_2 +m_1s_3t_2^2+ s_3t_2^3)=0.
$$
 The exceptional locus $\Theta_3$ is a smooth conic,
 intersecting the strict transform of 
 $\Theta_2$ at the two rationals points $s_3=0$.
 It intersects the strict transform of $ \Theta_0$
 at the   rational point $t_3=0$.  As $t=s^3t_3$, $\Theta_3$ has multiplicity $3$.
 We finished solving the singularity $p_0$.
We represent the three blowing-ups above in the picture below. 
  \begin{center} 
  \includegraphics[scale=0.3]{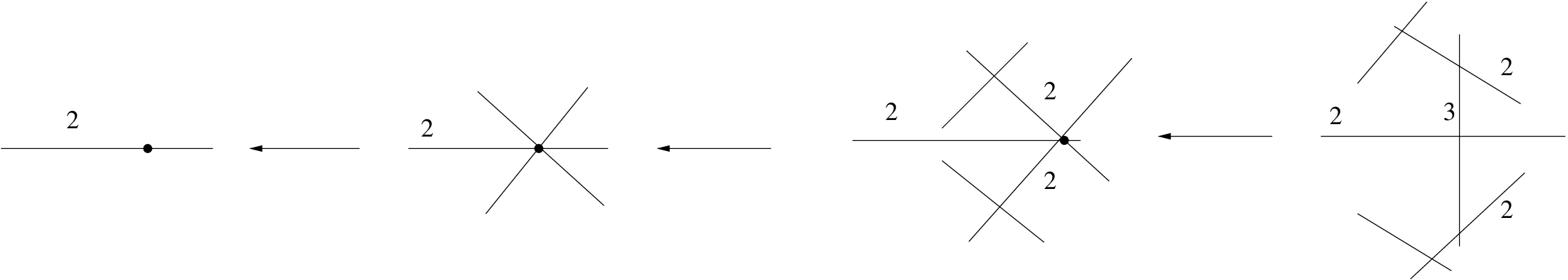}
  \end{center}
  In the last step all components are isomorphic to $\PP^1_{\F_2}$ and the
  intersection points are rational and transverse.  
  \medskip
  
 (III) The resolution of $p_*$ is obtained by one single 
blowing-up. The exceptional locus consists in a copy of $\PP^1_{\mathbb F_2}$
of multiplicity $1$, meeting transversely the strict transform of $\Theta_0$
at a rational point. 

 The complete resolution of the projective
 Weierstrass model defined by Equation~\eqref{eq:init}  is of
 type $[{\tt IV}^*-{\tt I}^*_{1}-\alpha]$ in \cite{NU}, page 175.  
\medskip

(IV) Suppose from now on that $\boxed{n=2n_1 \text{\ is even}}$. We already solved $p_{\infty}$ in Part (I). It remains to solve the singularity at $p_0$. We have 
$\lambda_{p_0}(Y)=3$ given by Equation~\eqref{eq:init}.
\smallskip 

(IV.1) \emph{First normalized blowing-up.} We use Equation~\eqref{eq:init}.
Let $Y_1\to Y$ be the normalized 
blowing-up along $p_0$. The equation over $D_+(\tilde{s})$
is
\begin{equation} \label{eq:eq9} 
y^2_1+st_1(s^2+n_1t_1)y_1+t_1(3s^3 +n_1s^3t_1 +m_1s^2t_1 + st_1^2)=0.
\end{equation}
The point at infinity, $s_1=0$ in $D_+(\tilde{t})$,
is regular in $Y_1$.  The exceptional locus $\Theta_1$
is $2\PP^1_{\mathbb F_{2}}$. It intersects the strict transform of $\Theta_0$ at
$p_1=(s=t_1=y_1=0)$, with $\lambda_{p_1}(Y_1)=4$. 
\smallskip

(IV.2) \emph{Normalized blowing-up along $p_1$.} We use Equation~\eqref{eq:eq9}. Over $D_+(\tilde{s})\subset Z_1$, we have $t_1=st_2$ and 
\begin{equation} \label{eq:eq8} 
y_2^2+st_2(s+n_1t_2)y_2+t_2(3 +n_1st_2 +m_1t_2 +t_2^2)=0.
\end{equation} 
The point at infinity $s_2=0$ ($s=t_1s_2$), is regular in $Y_2$ 

The exceptional locus $\Theta_2$ is a rational curve over $\F_2$ of
multiplicity $2$, with a cusp $p_2$ at $t_2=1$. 
The intersection of $\Theta_2$ with the strict transform of
$\Theta_0$ is $t_2=0$, and with that of $\Theta_1$ at $s_2=0$. 
The strict transforms of $\Theta_0$ and of $\Theta_1$ do not intersect each
other. 
 
The only possible singular point of $Y_2$ lying over $p_0$ is $p_2$.
Let us  compute $\lambda_{p_2}(Y_2)$. We make  the change of variables 
$t_2=x+1$ and $y_2=z_2+1$. We have 
\begin{equation} \label{eq:eq11}
z_2^2 + n_1s(1+x)^2z_2+x^3-ms^4(1+x)^4=0. 
\end{equation}

The strict transform of $\Theta_1$ in $Y_2$ is 
a projective line over $\F_2$, with self-intersection $-1$.
It is an exceptional divisor.
\smallskip

(IV.3.1) Suppose $\boxed{n_1 \text{\ is odd}}$. Then $\lambda_{p_2}(Y_2)=2$. 
After blowing-up $p_2$, the exceptional locus
$\Theta_3$ is a union of two projective lines over $\F_2$, of multiplicity $2$,
intersecting each other at a rational 
point $p_3$. It intersects the strict transform of $\Theta_2$ at $p_3$.  Moreover the point $p_3$ is regular in $Y_3$. The minimal regular model over $\Z_2$ of 
the genus $2$ curve defined by Equation~\eqref{eq:init} 
has type $[2{\tt IV}-2]$ in \cite{NU}, page 165.
  \begin{center} 
  \includegraphics[scale=0.4]{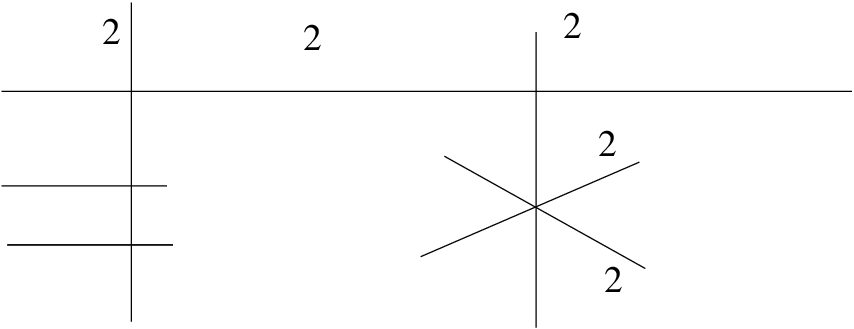}
  \end{center}  

\smallskip

(IV.3.2) Suppose $\boxed{n_1=2n_2 \text{\ is even}}$, then $\lambda_{p_2}(Y_2)=3$. We have 
an equation
\begin{equation}\label{eq:eq12}
z_2^2 + n_2s^3(1+x)^3z_2+x^3-ms^4(1+x)^4=0. 
\end{equation}
with $p_2=(s=x=z_2=0)$. The blowing-up $Y_3\to Y_2$ along $p_2$ contains
a unique singular point $p_3\in D_+(\tilde{s})$ defined by $s=x/s=z_2/s=0$,
with multiplicity $\lambda_{p_3}(Y_3)=4$.
The exceptional locus $\Theta_3$ is isomorphic to $\PP^1_{\F_2}$
and has multiplicity $4$.    
\smallskip

(IV.3.3) For the normalized blowing-up $Y_4\to Y_3$ along $p_3$,
the equation of $Y_4$ over $D_+(\tilde{s})\subset Z_4$ is
$$
z_3^2 + n_2s^2(1+sx_1)^3z_3+x_1^3s-ms^2(1+sx_1)^4=0. 
$$
Blow-up-normalize $Y_4\to Y_3$ along $p_3$ gives an equation
over $D_+(\tilde{s})$
$$
z_4^2 + n_2(1+s^2x_2)^3z_3+x_2^3-m(1+s^2x_2)^4=0. 
$$
If $\boxed{n_2 \text{ \ is odd}}$  (so $4$ divides exactly $n$), then
the closed fiber of the Lipman's resolution of the initial Weierstrass model is
 
  \begin{center} 
  \includegraphics[scale=0.4]{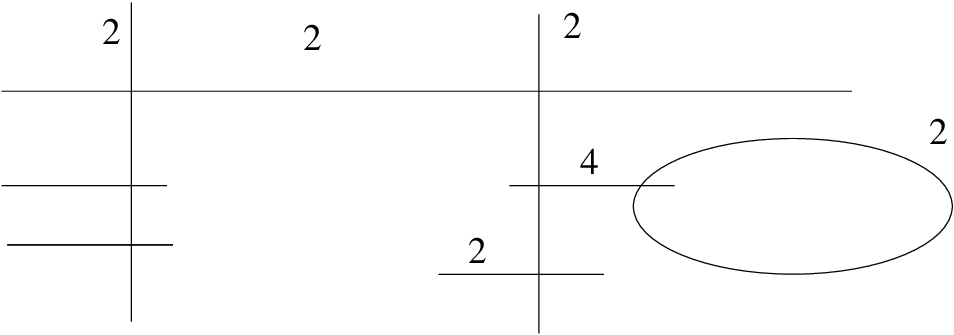}
  \end{center}  

\noindent where the last (elliptic) component is given by the equation 
$$
w^2+w+x^3+\bar{m}=0
$$
over $\F_2$. After contracting the strict transforms of $\Theta_1$ and $\Theta_3$, we find
the minimal regular model which has
type $[2{\tt I}_{0-2}]$ in \cite{NU}, page 159,  
  \begin{center} 
  \includegraphics[scale=0.4]{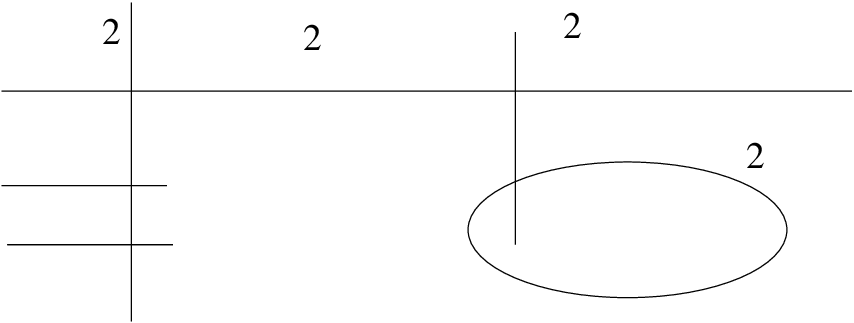}
  \end{center}  
This type is denoted by $[2{\tt I}_0-3]$ in \cite{LC}, pp. 71-72,   
for compatibility with the more general types $[2{\mathcal K}-\ell]$. 
\end{example} 

\begin{remark} The genus $2$ curve in the above example is that
  of \cite{LC}, Exemple 1 page 74. 
  But the resolution of singularities there is incorrect. 
  Nevertheless, the integer 
  $d$ defined in {\it op. cit.}, Th\'eor\`eme 1, page 70, is equal to
  $5$ in Case (IV.3.3). It still satisfies the inequalities  stated at the beginning  
  of that example.  

  To find the Artin conductor of the Jacobian of our genus $2$ curve, we need 
  not only the minimal resolution  $X\to Y$, but also the integer $d$ mentioned
  above. The latter can be read from the closed fiber of the last $Z_i$
  when there is no more singularities.
    We will come back to this question in \cite{LR}. 
\end{remark}

\subsection{Simultaneous resolution of singularities}
 Let $f : X\to Y$ be a finite morphism of integral noetherian
excellent surfaces. S. Abhyankar \cite{Ab}
asked for the existence of  resolutions of singularities
 $\tilde{X}\to X$ and  $\tilde{Y}\to Y$ such that $f$ extends to a finite
morphism $\tilde{X}\to \tilde{Y}$.  He proved that the answer is no if
 $\deg f >3$, and is positive when $f$ is a Galois cover of degree
 $\le 3$ of algebraic surfaces over a field of characteristic
 different from $\deg f$. For normal surfaces fibered over a discrete
 valuation ring, similar results are obtained in \cite{LL}, \S 7
 without assumption on the residue characteristics.

 \begin{corollary}\label{sim} Let $f : X\to Y$ be a finite morphism of degree $2$
   of integral noetherian excellent surfaces.
   Then there are resolutions of singularities $\tilde{X}\to X$ and
   $\tilde{Y}\to Y$ such that $f$ extends to a finite morphism
   $\tilde{X}\to \tilde{Y}$. 
 \end{corollary}

 \begin{proof} Replacing $Y$ with a resolution of singularities of $Y$ and
   then $X$ by the normalization of $Y$ in $K(X)$, we can assume that $Y$ is
   regular. Then $X\to Y$ is a normal double  cover
   (Example~\ref{exd_2}).  Applying Theorem~\ref{desing-main} to
   $X\to Y$ then  yields a simultaneous resolution of singularities. 
 \end{proof}

\end{section}

\begin{section}{Variations of the multiplicities after normalized
    blowing-ups} \label{control-l} 

In this section we will give a bound on the multiplicities of the points of
$\tilde{Y}$ lying over $E$ in terms of $\lambda_{q_0}(Y)$
(Theorem~\ref{lambda-sum}). It says roughly that the sum of the new
multiplicities is bounded by $\lambda_{q_0}(Y)$. 
This is done by comparing them to the multiplicities of the double
cover $\tilde{Y}\times_{Z'} E \to E$  when
$\tilde{Y}\times_{Z'} E$ is reduced, or of a ``twisted'' version of $\lambda_q$
otherwise.

\subsection{Double covers of \texorpdfstring{$\PP^1_k$}{PP1k}}  We first give some properties of the multiplicities $\lambda_q$ on double covers of  regular curves.

\begin{proposition} \label{genus-formula} Let $E$ be an integral regular curve over a field $k$. Let 
  $\Theta\to E$ be a double cover ($\Theta$ may be reducible) and let $\pi: \tilde{\Theta}\to \Theta$ be the normalization map. 
\begin{enumerate}[\rm (1)] 
\item Let $q\in E$ be a closed point $q\in E$. We have 
  $$
\dim_k \left( (\pi_*\cO_{\tilde{\Theta}})_q/\cO_{\Theta,q}\right)=[\lambda_q(\Theta)/2]. 
$$
\item Let $p\in \Theta$ be a point lying over $q$. If $\lambda_q(\Theta)$ is odd, then
  there is only one point $\tilde{p}\in \tilde{\Theta}_{p}$, and we have $k(p)=k(\tilde{p})$. 
 \item If $E$ is proper over $k$, then
 $$
p_a(\Theta)=p_a(\tilde{\Theta}) + \sum_{q\in E} [\lambda_q(\Theta)/2] [k(q): k], 
$$
where $p_a(C)=1-\chi(\cO_C)$ denotes the arithmetic genus of any proper variety
$C$ over $k$. 
\end{enumerate}
\end{proposition}

\begin{proof} (1)-(2) Similar to \cite{LTR}, Lemme 6.  The proof of
  (3) is similar to \cite{LB}, Proposition 7.5.4. 
\end{proof}

Next we consider the case $E=\PP^1_k$.  Fix an integer $r\ge 1$ and
two homogeneous polynomials $a(s, t), b(s, t)\in k[s, t]$ 
such that
\[ \max\{ 2\deg a, \deg b\}=2r.\] 
Consider the curve $ \Theta$
in the weighted projective space $\PP(1,1,r)$ defined by the equation
$$z^2+a(s,t)z+b(s,t)=0.$$  

\begin{lemma} \label{lambda-t-sum} Let $\Theta$ be as above and suppose that $\Theta$ is reduced.
  Denote by $\tilde{\Theta}\to \Theta$ the normalization map. 
  \begin{enumerate}[\rm (1)]
  \item We have a finite flat morphism $\Theta\to \PP^1_k$ of degree $2$.
Moreover $H^0(\Theta, \cO_\Theta)=k$ and $p_a(\Theta)=r-1$.  
  \item If $\Theta\to \PP^1_k$ is generically \'etale  then
\[ 
\sum_{q\in \PP^1_k} \lambda_q(\Theta)[k(q):k] \leq 2r. 
\] 
\item Suppose that $\Theta\to \PP^1_k$ is purely inseparable.
  \begin{enumerate}[\rm (a)]
  \item If $H^0(\tilde{\Theta}, \cO_{\tilde{\Theta}})=k$, then
\[ 
\sum_{q\in \PP^1_k} 2[\lambda_q(\Theta)/2][k(q):k] \le 2r-2.  
\]  
\item Otherwise,   we have 
\[ 
\sum_{q\in \PP^1_k} 2[\lambda_q(\Theta)/2][k(q):k]=2r, 
\] 
and $\lambda_q(\Theta)$ is even if $k(q)=k$. 
  \end{enumerate}
\item In all cases, $\lambda_q(\Theta)\le 2r$. 
  \end{enumerate} 
\end{lemma} 

\begin{proof} (1) The first part is clear and the second part can be proved by deforming the curve to a hyperelliptic curve of genus $r-1$, or by computing the {\v C}ech cohomology of $\cO_\Theta$ with the affine covering
  $\{ D_+(s), D_+(t)\}$. 

  (2) Denote by $\Delta(t)=\Delta_{\Theta/\PP^1_k}\in k[t]$. 
  For any closed point $q\in \Spec k[t]$,
  $\lambda_q(\Theta)\le \ord_q (\Delta)$ by Proposition~\ref{lambda-f}(2). 
  We have
\[ 
\sum_{q\in \Spec k[t]} \lambda_p(\Theta)[k(q):k] \le \deg \Delta(t)\le 2r.  
\] 
We are done if $t=\infty$ corresponds to a regular point. Otherwise  we have
$\lambda_\infty(\Theta)\le \deg t^{2r}\Delta(1/t) =2r-\deg \Delta(t)$. 

(3.a) We have $p_a(\tilde{\Theta})=\dim H^1(\tilde{\Theta}, \cO_{\tilde{\Theta}})\ge 0$. The desired
inequalities results from Proposition~\ref{genus-formula}(3) and (1). 

(3.b) Let $\tilde{k}=H^0(\tilde{\Theta}, \cO_{\tilde{\Theta}})$ with $k\subsetneq \tilde{k}$. Then $\tilde{\Theta}$ is a $\tilde{k}$-scheme. The morphism $\tilde{\Theta}\to \Theta\to \PP^1_k$ induces a morphism $\tilde{\Theta}\to \PP^1_{\tilde{k}}$ :
$$  \xymatrix{
\tilde{\Theta}  \ar[r] \ar[d] & \PP^1_{\tilde{k}} \ar[d]\\ 
\Theta    \ar[r]                & \PP^1_{k}.}
$$
We see that $\tilde{\Theta}\to \PP^1_{\tilde{k}}$ has degree
  $2/[\tilde{k}:k]$, thus $[\tilde{k}:k]=2$, and $\tilde{\Theta}\to \PP^1_{\tilde{k}}$ is
  birational hence an isomorphism.   
Therefore $p_a(\tilde{\Theta})=-1$ and 
$$
\sum_{q\in \PP^1_k} [\lambda_q(\Theta)/2][k(q):k]=r 
$$
by Proposition~\ref{genus-formula} and (1). Note that $\tilde{k}/k$ is inseparable because
$k(\Theta)=k(\tilde{\Theta})=\tilde{k}(t)$ is inseparable over $k(t)$.  

Let $k(q)=k$ and suppose that $\lambda_q(\Theta)$ is odd. Let $p\in \Theta_{q}$ and
let $\tilde{p}\in \tilde{\Theta}_{p}$. Then $k(p)=k(q)=k$ because $p$ is singular,
and $k(\tilde{p})=k(p)=k$ by Proposition~\ref{genus-formula}(2). But
$\tilde{k}=H^0(\tilde{\Theta}, \cO_{\tilde{\Theta}}) \subseteq k(\tilde{p})$, contradiction.

(4) is an immediate consequence of (2) and (3). 
\end{proof}

\subsection{Multiplicity along a regular closed subscheme}\label{mu-hyp}

Let $\psi : Y\to Z$ be a normal double cover of an integral noetherian regular
excellent scheme. Let $\Gamma$ be an
integral regular closed subscheme of $Z$. We define in this subsection
a multiplicity $\lambda_{\Gamma, q}(Y)$ which is just
$\lambda_q(Y\times_Z \Gamma)$ if $Y\times_Z \Gamma\to \Gamma$ is
a double cover (equivalently if $Y\times_Z \Gamma$ is reduced). If
this condition is not satisfied, we have to ``twist'' $Y\times_Z
\Gamma$ to get a double cover. 

\begin{lemma} \label{lambda-s} Let $A$ be integral noetherian regular
  and excellent. 
Let $B$ be a normal and free $A$-algebra of rank $2$. 
Suppose that $A\to B$ has ramification index $2$ above some 
prime ideal $\q\subset A$ of height $1$, and that $A/\q$ is normal.   
\begin{enumerate}[\rm (1)]
  \item There exists  a generator $y$ of $B$ such that
 $a\in  \q$, $b\in \q\setminus \q^2$.
  \item  Fix an $y$ given by (1). If $\m\in V(\q)$ is a maximal ideal
    of $A$ such that $A_\m/\q A_\m$ is regular, then  
  $\lambda_\m(B)=\lambda_\m(y-c)$ for some $c\in \q$ (depending in $\m$).
\end{enumerate} 
\end{lemma}

\begin{proof} (1) Let $y$ be a generator of $B$. As
  $A\to B$ has ramification index $2$ above $\q$,
  $B/\q B\simeq (A/\q)[T]/(T^2+\bar{a}T+\bar{b})$ is
  non-reduced. So $T^2+\bar{a}T+\bar{b}$ is a square in
  $\Frac(A/\q)[T]$. But $A/\q$ is integrally closed, so
  $T^2+\bar{a}T+\bar{b}$ is in fact a square in $(A/\q)[T]$. Therefore
  translating $y$ by a suitable element of $A$ we have
  $a, b\in \q$. Moreover, $b\notin \q^2$ because $B\otimes_A A_\q$ is normal.

(2) Suppose that for some $c_1\in A$ we have 
\[
\lambda_\m(y-c_1) = \min \{ 2\ord_\m(a+2c_1), \ord_\m(b+(a+2c_1)c_1-c_1^2) \} > \ell:=\lambda_\m(y). 
\] 
This implies that $\ord_\m(c_1)\geq \ell/2$ and $\ord_\m(b-c_1^2)>\ell$. 
Therefore
\[ c_1^2\in b+\m^{\ell+1}\subset \q + \m^{\ell+1}. \]
In $A/\p$, we have $\bar{c}_1^2\in \bar{\m}^{\ell+1}$. As $A/\p$ is
regular at $\bar{\m}$, this implies that
$\ord_{\bar{\m}}(\bar{c}_1)\geq r:=\lceil(\ell+1)/2 \rceil$. 
In other words $c_1=c+s$ with $c\in \q$ and $s\in\m^r$.
As $r>\ell/2$, we have $\lambda_\m(y-c)=\lambda_\m((y-c_1)+s) >
\ell$. 
Repeating this operation if necessary we will get  a generator
$y-c$ as desired, because $\lambda_\m(B)$ is finite.  
\end{proof} 

\begin{definition} \label{twist} Let $\psi : Y\to Z$ be as in the
  beginning of this subsection. 
  Let $\Gamma\subset Z$ be an integral regular closed subscheme of
  codimension $1$  such that 
  $\psi$ has ramification index $2$ above the generic point of
  $\Gamma$.   Then 
$\Theta:=Y\times_Z \Gamma \to \Gamma$ is finite flat of rank $2$, but 
$\Theta$ is not reduced. So we can't define the multiplicity
$\lambda_q(\Theta)$ as in Definition~\ref{def-lambda}. Instead we will define
a ``twisted'' version for $\Theta$. 

  We construct an effective divisor $D$ on $\Gamma$ as follows.
  Let $U\subseteq Z$ be an affine open subset such that
  $\cO_Z(-\Gamma)|_U=s\cO_Z(U)$ is free, and that $\psi^{-1}(U)\to U$ is free, defined by some equation
  $$
y^2+ay+b=0
  $$
such that $a, b\in s\cO_Z(U)$ and $b\notin s^2\cO_Z(U)$
(Lemma~\ref{lambda-s}(1)). Note that any another similar equation of
$\psi^{-1}(U)$ will give another $b'=ub+s^2c$ with $u$ a unit and
$c\in \cO_Z(U)$.  Then define $D|_{U\cap \Gamma}$ as the zero divisor
of $(s^{-1}b)|_{U\cap \Gamma}$.  It does not depend on the choice of the
equation. This allows to define an effective divisor $D$ on $\Gamma$. 

Now for any $q\in \Gamma$ we put
$$
\lambda_{\Gamma, q}(Y)=\ord_q(\cO_\Gamma(-D)_q) < +\infty 
$$
(order at $q$ of the ideal $\cO_{\Gamma}(-D)_q\subseteq \cO_{\Gamma,q}$). 
It will be crucial for the proof of Theorem~\ref{lambda-sum}(1) 
that we can compute $\lambda_{\Gamma,q}(Y)$ with
\emph{the same equation} for all points in a suitable open subset of
$\Gamma$. 
\end{definition}

\begin{lemma} \label{lambda-delta} Let $Y\to Z$ and $\Gamma$
  be as above.   Let $\xi_\Gamma$ be the generic point of $\Gamma$, 
let $q\in \Gamma$  be a closed point and let $\Theta=Y\times_Z \Gamma$.  
\begin{enumerate}[\rm (1)]  
\item  If $Y\to Z$ has ramification index $2$ over $\xi_\Gamma$, then
 \[ \lambda_{\Gamma, q}(Y) \ge \lambda_q(Y)-1 \geq 0.\]   
\item \label{1-2} Suppose that $Y\to Z$ has ramification index $1$
  above $\xi_\Gamma$.  Then $\Theta\to \Gamma$ is a double cover
  and $\lambda_{q}(\Theta)\ge \lambda_q(Y)$.
\end{enumerate} 
\end{lemma}

\begin{proof} We can suppose that $Z=\Spec A$ with $(A, \m)$
  local. Observe that for any $\alpha\in A$,  
$$ 
  \ord_\m(\alpha)\le \ord_{\m}(\alpha|_\Gamma)
$$ 
where $\alpha|_\Gamma$ denotes the image of $\alpha$ in $\cO(\Gamma)$. 
\medskip

(1) We can suppose that $Z=\Spec A$ is local. Let $s\in \m$ be a generator of
the ideal $\cO_Z(-\Gamma)$. 
Let $y\in B$ be such that 
$\lambda_\m(B)=\lambda_\m(y)$ and $a, b\in sA$
(Lemma~\ref{lambda-s}(2)).  By definition 
  $$\lambda_q(Y)\le \ord_q(b) = \ord_q(b/s)+1\le
  \ord_{q}((b/s)|_\Gamma)+1= \lambda_{\Gamma, q}(Y)+1$$
  (the last equality comes from the definition of
  $\lambda_{\Gamma, q}(Y)$). On the other hands, as $a, b\in \m$, 
we have   $\lambda_q(Y)\ge 1$.  
  \smallskip
  
  (2)  Let $B=A[y]$ with $\lambda_q(Y)=\lambda_\m(y)$.
  The projection $\Theta\to \Gamma$ is finite flat of rank $2$.
  The scheme $\Theta$ is reduced because it is a local complete
  intersection (Lemma~\ref{CM-dc}) and reduced at the generic
  points. Hence $\Theta\to   \Gamma$ is a double cover. An equation of $\Theta$ is given by
  $$\bar{y}^2+\bar{a}\bar{y}+\bar{b}=0 \mod s$$ 
  Then
  $$\lambda_q(\Theta)\ge \min \{ 2\ord_q(\bar{a}), \ord_q(\bar{b}) \} 
  \ge \min \{ 2\ord_q(a), \ord_q(b)\} = \lambda_q(Y)$$
  and the lemma is proved. 
\end{proof}

\subsection{Upper bound on the multiplicities}

\begin{theorem} \label{lambda-sum} Let $Z$ be an integral noetherian
  excellent surface. Let $\psi: Y\to Z$ be a normal double cover.
  Let $p_0\in Y$ be a singular point, $q_0=\psi(p_0)$. 
  
Let $\phi : \tilde{Y}\to Z'$ be as in Theorem~\ref{desing-main} and
let $E\subset Z'$ be the exceptional divisor of  $Z'\to Z$.
Let $\Theta:=\tilde{Y}\times_{Z'} E$. 
\begin{enumerate}[{\rm (1)}]  
\item If $\lambda_{q_0}(Y)$ is odd, then 
  $\lambda_q(\tilde{Y})\ge 1$ for all $q\in E$,   and 
  $$\sum_{q\in E} (\lambda_q(\tilde{Y})-1)[k(q): k(q_0)]\le
  \sum_{q\in E} \lambda_{E, q}(\tilde{Y}) [k(q): k(q_0)]\le
  \lambda_{q_0}(Y).$$ 
(See \S~\ref{mu-hyp} for the definition of $\lambda_{E, q}(\tilde{Y})$.)
\item Suppose that $\lambda_{q_0}(Y)$ is even. Then $\Theta\to E$ is a
  double cover. 
\begin{enumerate}[{\rm (a)}]  
\item  If $\lambda_{q_0}(Y)=\ord_{q_0}(\Delta_{Y/Z})$,  
  then
\[ 
\sum_{q\in E} \lambda_{q}(\tilde{Y})[k(q): k(q_0)]\le
\sum_{q\in E} \lambda_{q}(\Theta)[k(q): k(q_0)]\le 
\lambda_{q_0}(Y).
\] 
\item   Otherwise, 
  \[
    2\sum_{q\in E} [\lambda_{q}(\tilde{Y})/2][k(q): k(q_0)]\le
    2\sum_{q\in E} [\lambda_{q}(\Theta)/2][k(q): k(q_0)]\le
    \lambda_{q_0}(Y).
\]
\item In both cases we have
  \[
    \lambda_q(\tilde{Y})[k(q):k(q_0)]\leq
    \lambda_q(\Theta)[k(q):k(q_0)]\leq 
    \lambda_{q_0}(Y)
 \]
  for all $q\in E$.  
\end{enumerate}
\end{enumerate}
\end{theorem}

\begin{proof} We can suppose that $Z=\Spec A$ is local. 
  Let $\m$ be the maximal ideal of $A$, $k=A/\m$, let
  $\lambda_\m(B)=\lambda_\m(y)$ and $r=[\lambda_\m(B)/2]$. 

  (1) The cover $\phi : \tilde{Y}\to Z'$ has ramification index $2$ at
  the generic point of $\Theta$ (Proposition~\ref{birational}(2)).
  Let $s, t$ be a system of generators of $\m$. We use the notation of
  Proposition~\ref{bl-reg-ufd}. By Lemma~\ref{lambda-s}(1), we can
  suppose that $a, b\in sA$.  The equation of 
  $\tilde{Y}\times_{Z'}D_+(\tilde{s})$ is
  $$
y_1^2+a_1y_1+b_1=0, \quad a_1=a/s^r, \ b_1=b/s^{2r}
  $$
(Proposition~\ref{lemma2.12}). As $\ord_\m(a)\ge r+1$, we have $a_1\in s \Ast$. 
Write
\[
b=\sum_{0\le i\le 2r+1} \beta_i t^is^{2r+1-i} + \epsilon_{2r+2}, \quad
\beta_i\in A, \ \epsilon_{2r+2}\in \m^{2r+2}.
\]
Then, if $t_1=t/s$, 
$$
b_1/s = \sum_{0\le i\le 2r+1} \beta_i t_1^i+ s c, \quad c\in \Ast
$$
Let $f(\bar{t}_1)=\sum_{0\le i\le 2r+1} \bar{\beta}_i \bar{t}_1^i\in k[\bar{t}_1]$
be the class of $b_1/s$ mod $s$.
By Lemma~\ref{lambda-delta}(1), for all $q\in E\cap D_+(\tilde{s})$, we have
$$1\le \lambda_q(\tilde{Y})\le \lambda_{E, q}(\tilde{Y})=\ord_q(f(\bar{t}_1)).$$ 
Therefore 
$$\sum_{q\in E\cap D_+(\tilde{s})}
(\lambda_q(\tilde{Y})-1)[k(q):k] \le \deg f(\bar{t}_1).$$
The only point missed is the rational point $\{ s_1=0\}$ in $D_+(\tilde{t})$, where $s_1=s/t$.
The same reasoning as above shows that
$$\lambda_{s_1=0}(\tilde{Y})\le \ord_{s_1=0}(\sum_{0\le i\le 2r+1} \bar{\beta}_i
\bar{s}_1^{2r+1-i}) =2r+1-\deg f(\bar{t_1}).$$ 
This implies the desired inequality on the sum running through $E$.
\smallskip

(2)  The fact that $\Theta\to E$ is a double cover is Lemma~\ref{lambda-delta}(2). 
Let $\tilde{A}=\oplus_{n\ge 0} \m^n$. Then
$$\tilde{A}\otimes_A k \simeq \mathrm{Gr}_\m (A) \simeq k[s, t]$$
because $A$ is regular. 
For any $f\in \m^n$, if we consider $f$ as a homogeneous element
$\tilde{f}\in \tilde{A}_n$, then
the image $\tilde{f}(q_0)$ of $\tilde{f}$ in $\tilde{A}\otimes_A k$
is $0$ if $\ord_\m(f)>n$, and is a homogeneous polynomial of degree $n$,
equal to the initial form of $f$ otherwise.   
By Remark~\ref{bl-nor-w}, $\Theta$ is the closed subscheme of
$\PP(1,1,r)$ defined by the equation
$$
w^2+ \tilde{a}(q_0)w+\tilde{b}(q_0)=0 
$$
where $a, b$ are respectively considered as elements of $\m^r$ and $\m^{2r}$.
This is a double cover of $E$ of arithmetic genus $r-1$. 
So (2) is an immediate consequence of Lemmas~\ref{lambda-t-sum} and \ref{lambda-delta}.
\end{proof}

\begin{corollary}\label{lambda-uc} Keep the assumption of
  Theorem~\ref{desing-main}. For any $i\ge 1$, denote by
  $E_i\subseteq Z_i$ the exceptional divisor of $Z_i\to Z_{i-1}$. 
Then for all $q\in E_i$, we have 
\[\lceil \lambda_q(Y_i)/2\rceil  \le \lceil
  \lambda_{q_0}(Y)/2\rceil.\]  
\end{corollary} 

\begin{proof}  By induction on $i$ and by using
  Theorem~\ref{lambda-sum} (1) and (2.c). 
\end{proof}

\begin{example}\label{lambda-34}   Let $Z=\Spec k[s,t]$ over a field $k$. Let $n\ge 3$.
Consider $Y\subset \Spec k[s,t, y]$ defined by 
$$y^2+t^3+s^n=0.$$
Let $q_0=(0,0)\in Z$. Then 
$\lambda_{q_0}(Y)=3$. We get a sequence of normalized blowing-ups
(for $i+1< n/3$)
$$
\to Y_{i+1} \to Y_{i} \to \dots \to Y_1 \to Y_0=Y 
$$
where $Y_i$ has an affine open subset defined by 
$$
y_{i}^2+s^{n- 3i+1}+st_{i}^3=0, \quad t_i=t/s^i,  y_i=y/s^{3m+1},
$$ 
if $i=2m+1$,
and by 
$$
y_{i}^2+s^{n-3i}+t_{i}^3=0, \quad t_i=t/s^i,  y_i=y/s^{3m},  
$$
if $i=2m$. The center of $Y_{i+1}\to Y_{i}$  is $p_{i}=\{ y_i=s=t_i=0\}$
which has $\lambda_{p_i}(Y_i)=4$ or $3$ depending on whether $i$ is
odd or even. There is no other singular points on $Y_i$. 
\end{example} 

\end{section}

\begin{section}{Desingularization algorithm} \label{algo}
  
Let $A$ be a noetherian regular excellent domain of dimension $2$.
We will also 
assume $A$ is factorial ({\it e.g.} local).  This assumption implies
that locally free $A$-algebras $B$ of rank $2$ are free (because
the quotient $A$-module $B/A$ is locally free of rank $1$
by Lemma~\ref{CM-dc}).
Let $B$ be an integral  free $A$-algebra of rank $2$.   We denote as
before $Z=\Spec A$ and $Y=\Spec B$. 

\subsection{Normalization}\label{sect:normalize}  

We determine the integral closure $B'$ of $B$ when
$\Delta_{B/A}\ne 0$.  
Write $B=A[y]$ with $y^2+ay+b=0$. 
Let $t$ be a prime divisor of $a^2-4b$. Denote by $\nu_t$  
the normalized valuation on $\Frac(A)$ induces by $tA$.

As $A$ (hence $B$) is excellent, $B'$ is finite over $B$.
By Serre's criterion $B'$ is  Cohen-Macaulay, hence
finite flat of rank $2$ over $A$. Therefore  
$B'$ has a basis $\{ 1, z \}$ over $A$.
There exist $u, v\in A$ such that
\begin{equation} \label{eq:zy}  
y=uz+v. 
\end{equation}
In particular $B=A[y]$ is normal at primes lying over $tA$ if and only
if $t\nmid u$.  Note that $u \mid a+2v, b+av+v^2$, hence
$u \mid b-v^2$. 
\medskip

\noindent {\bf Normalization algorithm} Let $y$ and $t$ be as above. 
\begin{enumerate}[\rm (1)]
\item \label{step1} Let $r=[(\min\{ 2\nu_t(a), \nu_t(b) \})/2]$.  Let $y_1=y/t^r$,
$a_1=a/t^r$, $b_1=b/t^{2r}$. Then $a_1, b_1\in A$ and 
$$
y_1^2+a_1y_1+b_1=0 
$$
with $B\subseteq B_1:= A[y_1]\subset L:=\Frac(B)$ and 
$\min\{ 2\nu_t(a_1),  \nu_t(b_1) \}\le 1$. 
\item Now replace $y$ with $y_1$,
we can assume $\min\{ 2\nu_t(a),  \nu_t(b) \}\le 1$.
Then $A[y]$ is normal at primes above $tA$ except possibly when 
$\nu_t(a)>0$, $\nu_t(b)=0$ and $\chara(A/tA)= 2$ (Lemma~\ref{normal-codim1}).   

\item Assume that the above conditions hold. By the discussions following
  Equation~\eqref{eq:zy}, we see that $A[y]$ is not normal at some prime
  lying over $tA$  if and only if there exists $c\in A$ such that
  $t  \mid b-c^2$ (equivalently $b$ is a square in $A/tA$).
  
\item Suppose that such a $c\in A$ exists. 
  Let  $y_1=y-c$. Then $A[y]=A[y_1]$ with  
$$
y_1^2+(a+2c)y_1+(b+ac+c^2)=0
$$ 
and $\nu_t(a+2c), \nu_t(b+ac+c^2) >0$.  
Go back to Step (\ref{step1}) with this new equation.
\end{enumerate}

When the above algorithm terminates (this always happens eventually because the
valuation $\nu_t(u)$ of the factor $u$ in Equation~\eqref{eq:zy} 
decreases strictly  when we pass through Step (4) twice.), we go to another
prime divisor of $a^2-4b$.  At the end we get $B'$ together with
a basis $\{ 1, z\}$.

\subsection{Finding the singular points}\label{sect:find_sing} 
We are given  a normal $B=A[y]$ with 
$$y^2+ay+b=0.$$ 

\noindent{\bf Singular points of $Y$.} Even though there are
only finitely many singular points, there is no obvious way to find them for
general $A$.  

 - If $A$ (hence $B$) is  a finite type algebra over some perfect field,
we can use Jacobian criterion.

- Suppose that $Z$ is smooth with integral fibers over some Dede\-kind 
domain $R$ ({\it e.g.} $A=R[x]$) and that $B\otimes_A \Frac(R)$ is
regular. We can look for singular points  in the closed fibers of
$Y\to \Spec R$ (which are curves over fields).  The singular points of
$Y$ are among them. 

However, if $B\otimes_R {k(w)}$ is non-reduced for some closed
point $w\in \Spec R$, all points of the fiber $Y_w$ are singular. The worst
situation is when $\Delta_{B/A}(w)=0$.
Suppose $R$ is local of dimension $1$ with a uniformizing element  $\pi$. 
Then $\pi \mid a$ and $b$ is a square in $A/\pi A$. Let $b\equiv c^2
\mod \pi $. Replace $y$ with $y+c$, then we get an equation with
$\pi \mid a, b$. As $B$ is normal, $\pi^2 \nmid b$.
The singular points of $Y$ are exactly those lying over the 
zeros of $b/\pi$  in the curve $\Spec (A\otimes_R k(w))$.  
\medskip

\noindent {\bf Singular points in the normalized blowing-ups.} Let
$q_0\in Z$ be such that $\lambda_{q_0}(Y)\ge 2$. Let $k=k(q_0)$. 

Let $\tilde{Y}\to Y$ be the normalized blowing-up at the point $p_0$ lying
over $q_0$ (see \S~\ref{normal-bl}). We search for singular points of 
$\tilde{Y}$ lying over  the exceptional locus 
$E\subset Z'\stackrel{\mathrm{bl}_{\m}}{\longrightarrow} Z$. 
Let $a_1, b_1\in \Ast$ be as in Proposition~\ref{lemma2.12}.
We first look for singular points lying over $D_+(\tilde{s})$. 
There are two cases.

(1) $s \nmid a_1^2+4b_1$ in $\Ast$  (ramification index $1$). 
Then any singular point of $\tilde{Y}$ lying over $E\cap D_+(\tilde{s})$
must correspond to a multiple factor $f(t_1)$ of the polynomial 
$\overline{a_1^2+4b_1} \in {\Ast}/(s)\simeq k[t_1]$. Lift
$f(t_1)$ to an $F\in \Ast$. Then $(s, F)$ is a maximal ideal. Compute
$\lambda_{(s, F)}(\Ast[y_1])$ to see whether we really get a singular point.  

(2) $s \mid a_1^2+4b_1$. Then we have $s \mid a_1$
(we take $a=0$ if $\chara(k)\ne 2$).

(2.1) If $s \mid b_1$ 
then $s^2 \nmid b_1$. The singular 
points correspond exactly to the zeros of
$\overline {s^{-1}b_1}\in \Ast/(s)\simeq k[t_1]$.  

(2.2) If $\chara(k)=2$, and 
$\bar{b}_1\in \Ast/(s)$
is a square. Translate 
$y_1$ by an element of $\Ast$ to make $s \mid b_1$. Go back to (2.1). 

(2.3) If $\chara(k)=2$, and  $\bar{b}_1\in \Ast/(s)$ 
is  not a  square. 
A necessary condition for $(s, F)$, $F\in \Ast$, to define a
singular point is that $\bar{F}(t_1)\in k[t_1]$ divides the derivative
of $\bar{b}_1(t_1)$. 
\medskip 

(3) It remains the points lying over $V_+(\tilde{s})\cap E$. 
Let $t\in \m$ be such that $s, t$ generate $\m A_\m$. Then
$V_+(\tilde{s})\cap E=V_+(\tilde{s})\cap D_+(\tilde{t})\cap E$ is the
point $s_1=t=0$ in $D_+(\tilde{t})$ (where $s_1=s/t$). Check
similarly whether it induces a singular point of $\tilde{Y}$ .

\subsection{Desingularization algorithm}\label{desing-algo}
Let $Z$ be an integral noetherian excellent regular surface. Let $Y\to Z$ be a double cover. We work locally on $Z$ and suppose $Z=\Spec A$ is affine. 

(1) Find the normalization of $Y$ (\S~\ref{sect:normalize}),
then suppose $Y$ is normal.  

(2) Find the singular points of $Y$ (\S~\ref{sect:find_sing}). 

(3) Fix a singular point lying over a closed point $q_0\in Z$.
Compute $\lambda_{q_0}(Y)$ using Lemma~\ref{lambdaz}.

(4) Consider the normalized blowing-up $\tilde{Y}\to Y$ along the
point $p_0$ lying over $q_0$. Use the affine equation of $\tilde{Y}$
given Proposition~\ref{lemma2.12} to find the singular points of $\tilde{Y}$
lying over $p_0$.  

(5) Repeat the above process for all singular points of
$Y$. At the end we get a new  normal double cover of an integral
noetherian excellent regular surface. 

(6) Go back to (2).

\subsection{Presentation of a regular ring and of its Rees algebras}  The most
computationally complicated part of our algorithm should be the
determination of $\ord_\m$ and of the initial form of an element of $A$
(Lemma~\ref{lambdaz}(2.b)). 
For these purposes we need an as simple as possible 
presentation (by generators and relations) of $A$ 
and of the rings appearing after blowing-up $\Spec A$ along
maximal ideals. 

We will restrict ourselves to the case when $A$ is of finite type over
some principal ideal domain $R$
({\it e.g.} $\Z$, $\mathbb F_p[T]$).
Write  
$$A=R[U_1, \dots, U_n]/I $$
for some $n\ge 2$.
Let $\m \subset A$ be a maximal ideal of height $2$ generated by $s, t$.
Then
$$\Ast=A[T_1]/(sT_1-t) = R[U_1, \dots, U_n, T_1]/(I,ST_1-T)$$ 
where $S, T\in R[U_1, \dots, U_n]$ are liftings of $s, t$ respectively.
Thus in general we need to add a new variable and a new relation to represent
the $R$-algebra $\tilde{A}_{(\tilde{s})}$. However this is not necessary in some
particular cases.

Suppose that we can write 
$$A=A_0[S, T]/I $$
for some $R$-algebra $A_0$
such that $\m$ is generated by the classes $s, t$ of $S, T$ in $A$.
Then 
$$\Ast=A_0[S, T, T_1]/(I, ST_1-T)=A_0[S, T_1]/I'$$ 
where $I'$ is obtained by replacing $T$ with $ST_1$ in the polynomials
in $I$.
Similarly 
$$\tilde{A}_{(\tilde{t})}=A_0[S, T, S_1]/(I, TS_1-S)=A_0[T, S_1]/I''$$ 
where $I''$ is obtained by replacing $S$ with $TS_1$ in the polynomials
in $I$. 

Let $\m_q$ be a maximal ideal of $\Ast$ containing $s$. It corresponds
to a closed point $q$ of $E\cap D_+(\tilde{s})$.  Suppose that $q$ is
rational over $k:=A/\m$. Then there exists  $\theta\in A_0$ such that
$\m_q$ is generated by the classes of $S, T_1-\theta$. After translating $T_1$
by $\theta$, we can suppose $\theta=0$. Then 
the blowing-up of $\Spec \Ast$ along $q$ is covered by the spectra of   
$$
A_0[S, T_2]/J, \quad A_0[T_1, S_2]/J' 
$$
where $J$ is obtained by substituting $ST_2$ to $T_1$ in the
polynomials of $I'$, and $J'$ is obtained by substituting $T_1S_2$ to $S$
in the polynomials of $I'$.

\begin{corollary}  Let $R$ be a PID. 
  Suppose $A=R[x]$. Let $\pi\in R$ be an irreducible element and let $k=R/(\pi)$.   
  If we blow-up $\Spec A$ along a rational point $q$ of $\Spec k[x]$,
  and then successively  
  blow-up $k$-rational points in the exceptional divisors, then all
  intermediate schemes are covered by regular affine schemes
  of the form
  $$\Spec (R[S,T]/(f(S,T)-\pi))$$ with $f(S,T)\in (S,T)R[S,T]$. 
\end{corollary}

\begin{proof} We can suppose that $q$ corresponds to the maximal ideal
  $(x, \pi)\subset R[x]$. Let $A_0=R$. We have $A=A_0[S,T]/(f(S, T)-\pi)$ with
  $f(S,T)=T$.  
  Then apply the above discussions.
\end{proof}

\end{section}

{Univ. Bordeaux, CNRS, IMB, UMR 5251, F-33405 Talence, France}

\tt{qing.liu@math.u-bordeaux.fr}

\end{document}